\documentclass{article}
\usepackage{graphicx} 
\usepackage{mathtools}
\usepackage{amsmath}
\usepackage{amsthm}
\usepackage{amssymb}
\usepackage{todonotes}
\usepackage{hyperref}
\usepackage{tikz-cd}
\usepackage{longtable}
\usepackage[backend=biber, style=numeric, sorting=nyt]{biblatex}
\DeclareFieldFormat{pages}{#1}
\renewbibmacro{in:}{%
  \ifentrytype{article}
    {}
    {\bibstring{in}%
     \printunit{\intitlepunct}}}
\usepackage{graphicx}
\usepackage{array}
\graphicspath{ {figures/} }
\newcounter{statement}[section]
\renewcommand\thestatement{\thesection.\arabic{statement}}

\newenvironment{proposition}[1][]{\refstepcounter{statement}\vspace{12pt}\noindent\textbf{Proposition~\thestatement.\ }\label{prp:#1}\rmfamily\begin{itshape}}{\end{itshape}}

\newenvironment{theorem}[1][]{\refstepcounter{statement}\vspace{12pt}\noindent\textbf{Theorem~\thestatement.}
\rmfamily\begin{itshape}}{\end{itshape}}

\newenvironment{conjecture}[1][]{\refstepcounter{statement}\vspace{12pt}\noindent\textbf{Conjecture~\thestatement.}\label{con:#1}\rmfamily\begin{itshape}}{\end{itshape}}

\newenvironment{corollary}[1][]{\refstepcounter{statement}\vspace{12pt}\noindent\textbf{Corollary~\thestatement.}\label{cor:#1}\rmfamily\begin{itshape}}{\end{itshape}}

\newenvironment{example}[1][]{\refstepcounter{statement}\vspace{12pt}\noindent\textbf{Example~\thestatement.}\label{ex:#1}\rmfamily\begin{itshape}}{\end{itshape}}

\newenvironment{lemma}[1][]{\refstepcounter{statement}\vspace{12pt}\noindent\textbf{Lemma~\thestatement.}\label{lem:#1}\rmfamily\begin{itshape}}{\end{itshape}}

\newenvironment{problem}[1][]{\refstepcounter{statement}\vspace{12pt}\noindent\textbf{Problem~\thestatement.}\label{con:#1}\rmfamily\begin{itshape}}{\end{itshape}}

\usepackage[backend=biber, style=numeric, sorting=nyt]{biblatex}
\renewbibmacro{in:}{%
  \ifentrytype{article}
    {}
    {\bibstring{in}%
     \printunit{\intitlepunct}}}

\allowdisplaybreaks

\newcommand{\A}{\mathbb{A}}
\newcommand{\B}{\mathbb{B}}
\newcommand{\C}{\mathbb{C}}
\newcommand{\D}{\mathbb{D}}
\newcommand{\E}{\mathbb{E}}
\newcommand{\F}{\mathbb{F}}

\newcommand{\G}{\mathbb{G}}
\newcommand{\I}{\mathbb{I}}

\newcommand{\K}{\mathbb{K}}
\newcommand{\M}{\mathbb{M}}
\newcommand{\N}{\mathbb{N}}
\renewcommand{\O}{\mathbb{O}}
\renewcommand{\P}{\mathbb{P}}
\newcommand{\Q}{\mathbb{Q}}
\newcommand{\R}{\mathbb{R}}
\renewcommand{\S}{\mathbb{S}}
\newcommand{\T}{\mathbb{T}}
\newcommand{\U}{\mathbb{U}}
\newcommand{\V}{\mathbb{V}}

\newcommand{\X}{\mathbb{X}}
\newcommand{\Y}{\mathbb{Y}}
\newcommand{\Z}{\mathbb{Z}}

\DeclareUnicodeCharacter{27F9}{\implies}

\usepackage[backend=biber, style=numeric, sorting=nyt]{biblatex}
\DeclareFieldFormat{pages}{#1}
\renewbibmacro{in:}{%
  \ifentrytype{article}
    {}
    {\bibstring{in}%
     \printunit{\intitlepunct}}}
\usepackage{graphicx}
\usepackage{array}
\graphicspath{ {figures/} }

\usepackage{dcolumn}
\newcolumntype{2}{D{.}{}{2.0}}
\title{Extension Monads: Some Structure Theorems}
\author{Danielle Bowerman, Matt Insall}
\date{January 2025}

\begin{document}

\maketitle

\begin{abstract}
    We investigate the behavior of extension monads, introduced in the 1990s by the second author, in terms of structure results for infinitely many finitary operations and common constructions in varieties or categories of algebras. Specifically, we see that for finite collections of algebras of the same signature there the extension monad operation commutes with the direct product operation, and that the only retractions from an enlargement or extension monad are trivial.
\end{abstract}

\section{Introduction}
This introduction will review the basics of nonstandard methods and enlargements, as well as universal algebra. Readers comfortable with these topics could skip to the subsection on extension monads, section 1.2. For additional details on these fundamentals, refer to \cite{Burris2012} and \cite{Loeb1985}.

Given a set of urelements $U$ and another set of urelements ${}^*U$, a \textbf{superstructure embedding}, where the image of an element is its \textbf{${}^*$-transform} with transformation $^*\_$, is an injective function, $U \longrightarrow {}^*U$, such that we can extend it to their superstructures, $V(U)$ and $V({}^*U)$, where $a \in b \iff {}^*a \in {}^*b$. The existence of a nontrivial superstucture embedding was demonstrated using {\L}o\'s's theorem by \cite{Wash}. Next, we outline some details of our construction of this nontrivial superstructure embedding, starting with filters. Filter $\T$ on indexing set $I$ is a subset of the power set of $I$, $\mathcal{P}(I)$, with the following properties: if $X,Y \in \rho$, then $X \cap Y \in \T$; and if $X \in \rho$ and $X \subseteq Y$, then $Y \in \rho$.
Filters can have several properties: a filter is \textbf{principal} if and only if it is of the form $\{Y|Y \supseteq X\}$ for some $X \in \mathcal{P}(I)$; \textbf{prime} if and only if when $A \cup B \in \rho$, we have $\rho \cap \{A, B\} \neq \emptyset$; and \textbf{proper} if it is a filter that is not the entire power set.
A filter is an \textbf{ultrafilter} if and only if it is a proper filter, i.e. it is not equal to $\mathcal{P}(I)$, and it is maximal among proper filters. We will use ultrafilters in our construction. Equivalently, filter $\rho$ is an ultrafilter if and only if it is proper and, for every set $A \in \mathcal{P}$, either $A \in \rho$ or $I \setminus A \in \rho$.

The \textbf{ultrapower} of $(V(U), \in)$, defined by ultrafilter $\rho$, which is not principal or equivalently includes every cofinite subset of $I$, is $(V(U)^I, \in^I)\mathrel{/}{\sim} = (V(U)^I\mathrel{/}{\sim},{\mathrel{\in^I\mathrel{/}{\sim}})}$, where two vectors $\vec{a}\in V(U)^I$ and $\vec{b}\in V(U)^I$ are equivalent, written $\vec{a} {\sim} \vec{b}$, if and only if $\{i \in I| a_i = b_i\} \in \rho$. Additionally, $[\vec{a}]_{\sim} \mathrel{\in^I\mathrel{/}{\sim}} [\vec{b}]_{\sim}$ if and only if $\{i \in I| a_i \in b_i\} \in \rho$. Such an ultrapower is called an \textbf{enlargement} of $V(U)$ since $\rho$ being not principal means that $V(U)$ will not be isomorphic to $V({}^*U)$.
Once we form the ultrapower, the Mostowski Collapsing Lemma guarantees that membership modulo $\sim$ can be replaced by actual membership.

We will often refer to standard, nonstandard, internal, or external entities, elements, or structures, so we will review those terms here. A set or element of a superstructure enlargement is standard if it is the image under the superstructure embedding of some set or element of the original superstructure. An internal entity is any element of ${}^*A$ where $A$ is any element of $V(U)$. Any other subset of $V({}^*U)$ is external.

A natural and ``canonical'' example of the difference between an internal and external set is the natural numbers, $\N$. Here, $\N$ is external and is not a member of its superstructure enlargement ${}^*V({}^*U)$. The internal counterpart to $\N$ is the set of hyperfinite numbers ${}^*\N$.

On the other hand, the ${}^*$-transform of the collection of finite subsets of $A$ is the collection of hyperfinite subsets of ${}^*A$. It is important to note that the cardinalities of these sets can be much larger than the cardinality of $A$, and hyperfiniteness just means that they behave like finite sets in relation to other internal sets.

The \textbf{transfer principle}, which allows us to map sentences about the standard superstructure to its nonstandard counterpart, follows from our definition of a superstructure embedding and {\L}o\'s's theorem. This theorem entails that a mathematical statement, $\sigma$, in first-order predicate logic about $V(U)$ holds if and only if statement ${}^*\sigma$ holds in $V({}^*U)$, where ${}^*\sigma$ is the same statement as $\sigma$, except that each constant element $S$ of $V(U)$ utilized in $\sigma$ is replaced with the star transform of that element, ${}^*S$. The transform of any individual finite set of urelements is essentially the same set.

In universal algebra, it is typical to denote an algebra as an ordered pair $\A = (A, \mathcal{O}_\A)$, where $A$ is the underlying set, which we take to be composed of urelements, and $\mathcal{O}_\A$ is an indexed collection of functions from $A$ to itself, which are called the operations of the algebra. Homomorphisms, the morphisms in the category of algebras, are functions $f$ from the underlying set $X$ of one algebra to another underlying set $Y$ of an algebra with the same language, such that for every $o \in \mathcal{O}$, $f(o_\X(z_1,z_2,...)) = o_\Y(f(z_1),f(z_2),...)$, independent of the arguments $z_j$. A partial homomorphism is a function whose domain is some subset $S$ of an underlying set $X$ of an algebra which satisfies the condition that for every $o \in \mathcal{O}, z_1,z_2,... \in S$, such that $o_\X(z_1,z_2,...) \in S$, we have $f(o_\X(z_1,z_2,...)) = o_\Y(f(z_1),f(z_2),...)$; as such, every homomorphism is a partial homomorphism that is total. Given standard algebra $\A = (A, \mathcal{O}_\A)$ and subset $F \subseteq {}^*A$, we can generate an internal subalgebra of ${}^*\A$. We define $A_0 = F$, and then for each $n \in {}^*\N$, we define $A_{n+1} = \{o(x_1,x_2,...,x_{\eta_o})|o \in {}^*\mathcal{O}_{{}^*\A} \text{ with arity } \eta_o; x_1,x_2,...,x_{\eta_o} \in A_n \}$, where ${}^*\mathcal{O}_{{}^*\A}$ is the collection of operations on ${}^*\A$, called internal operations, and $\eta_0$ is some element of ${}^*\N$, called a hyperfinite number, which can be larger than any finite number. We denote this algebra as $\check{F} = \displaystyle \bigcup_{n \in {}^*\N} \A_n$. The transfer principle applied to the standard case gives us closure, and so the proof that $\check{F}$ is an algebra is quite similar to the proof in the standard setting. In fact, it is just the standard proof after applying the transfer principle. In addition, for subalgebra $\S = (S, {}^*\mathcal{O)}$ of ${}^*\A$, we define the standard reduct of $\S$ to be $\S|_\mathcal{O} = (S; \{{}^*f_\S| f \in \mathcal{O}\}$.

Transforming $\A$, we get ${}^*\A = ({}^*A, {}^*\mathcal{O})$. If $\mathcal{O}$ is finite, then $|{}^*\mathcal{O}| = |\mathcal{O}|$. Three considerations prevent the simple equality ${}^*\mathcal{O} = \mathcal{O}$ from always being true. If the algebra is infinite, then the operations of the algebra have an infinite cardinality. Thus, transferring each operation extends it to cover the newer extended underlying set due to the latter containing its standard form. Another consideration is whether the operations are infinitary because infinitary operations extend their ordered collection of inputs to a nonstandard ordinal when transferred, thus preventing equality. Nevertheless, if $\mathcal{O}$ is infinite, then the indexing set of $\mathcal{O}$ is smaller than the indexing set of ${}^*\mathcal{O}$, where every indexed element of $\mathcal{O}$ gets extended to an element of ${}^*\mathcal{O}$. As a consequence, ${}^*\A$ is not directly comparable to $\A$ as an algebra, and it can have many new operations, some of which may have infinitely many inputs.

We briefly cover the concept of satisfaction as it relates to algebras, which is covered in more depth in \cite{Burris2012}, and look at how it relates to ${}^*$-transforming. Given some type $\mathcal{F}$, which is an ordered collection of ordered pairs of urelements representing operation names and associated arities, with least upper bound $\sigma$ to the arities, and a set $X$ of urelements we call variables, the \textbf{set $T(X)$ of terms of type $\mathcal{F}$} over $X$ is the smallest set that contains every nullary operation symbol and every member of $X$, and for all $t_1,t_2,...,t_n \in T(X)$ and $n$-ary $f\in\mathcal{O}$, which is the ordered collection of first members of $\mathcal{F}$, we also have that the expression $f(t_1,t_2,...,t_n)$ is in $T(X)$. The \textbf{term algebra}, $\T(X)$ has this set as its underlying set and its operations are of the form $f^{\T(X)}(t_1,t_2,...,t_n) = f(t_1,t_2,...,t_n)$ for each $n$-ary $f \in \mathcal{O}$. Given a term $t \in T(X)$, the \textbf{corresponding $n$-ary term mapping} $t_n^\A:\A^n\rightarrow\A$ is a mapping such that if $t$ is a variable $x_i$, then $t_n^\A(a_1,...a_n) = a_i$, and if $t$ is of the form $f(t_1,t_2,...,t_k)$ for $k$-ary $f \in \mathcal{O}$, we have that $t_n^\A(a_1,a_2,...,a_n) = f^\A({t_1}_n^\A(a_1,a_2,...,a_n),{t_2}_n^\A(a_1,a_2,...,a_n),...,{t_k}_n^\A(a_1,a_2,...,a_n))$. 

Now an \textbf{equation} of type $\mathcal{F}$ is an expression of the form $s \approx t$ for $s,t \in T(X)$. An algebra $\A$ \textbf{satisfies} an equation $s \approx t$ if and only if, for all ordinals $n \leq \sigma$, for every choice of elements $a_1,a_2,...,a_n \in A$, we have that $s^\A_n(a_1,a_2,...,a_n) = t^\A_n(a_1,a_2,...,a_n)$. Since we are concerned with algebras whose underlying sets are composed of urelements, we can define a relation $\models$ between algebras whose underlying sets are made of urelements and equations of type $\mathcal{A}$ where $\A \models e$ if and only if $\A$ satisfies $e$. Furthermore, since algebras of urelements have rank in the superstructure of at most 7 and equations have rank of at most 2 when viewed as an ordered sequence of urlement symbols, relation $\models$ will be an element of the superstructure of rank at most 10. This allows us to take the ${}^*$-transform of relation $\models$ to find a notion of ${}^*$-satisfaction, which is a relation between internal algebras and equations. By transfer, we see that if $\A\models e$, we have that ${}^*\A \ {}^* \hspace{-.14cm} \models {}^*e$ and $e = {}^*e$ as long as no operations in the equation have infinite arity because, in this case, $e$ will be a finite ordered sequence of urlements. By transfer, we see that for any internal subalgebra $\I$ of ${}^*\A$, we have $\I \ {}^* \hspace{-.14cm} \models e$, where $e$ is $s \approx t$, if and only if for all internal ordinals $\eta \leq {}^*\sigma$, for all choices of elements $a_1,a_2,...,a_\eta \in I$, we have that ${}^*s^\A_\eta(a_1,a_2,...,a_\eta) = {}^*t^\A_\eta(a_1,a_2,...,a_\eta)$, which are two internal term mappings. We see that ${}^*t_\eta^\A:\A^\eta\rightarrow\A$ is a mapping such that if $t$ is a variable $x_i$, then $t_\eta^\A(a_1,...a_\eta) = a_i$, and if $t$ is of the form $f(t_1,t_2,...,t_\kappa)$ for $\kappa$-ary $f \in {}^*\mathcal{O}$, we have that $t_\eta^\A(a_1,a_2,...,a_\eta) = f^\A({t_1}_\eta^\A(a_1,a_2,...,a_\eta),{t_2}_\eta^\A(a_1,a_2,...,a_\eta),...,{t_\kappa}_\eta^\A(a_1,a_2,...,a_\eta))$. However, if $f$ is an image of a finitary element of $\mathcal{F}$, we have that an internal term mapping of finite arity will be a standard term mapping. For standard equation $e$, $\I \ {}^* \hspace{-.14cm}\models e$, this shows us that $\I$ satisfies $e$ in the standard sense.

\begin{example}\label{}
The above allows us to see that an internal group is also a group in the standard setting. As a group has only finitely many finitary operations and, thus, satisfies only finite equations. Given type $\mathcal{G} = ((1,0),(\iota,1),(\boldsymbol\cdot,2))$ with the set of operations $\mathcal{O} = (1,\iota,*)$, we see that a group satisfies the following equations:
$x \boldsymbol\cdot 1 \approx 1 \boldsymbol\cdot x \approx x$
$x \boldsymbol\cdot \iota(x) \approx \iota(x) \boldsymbol\cdot x \approx 1$
$x\boldsymbol\cdot(y\boldsymbol\cdot z) \approx (x \boldsymbol\cdot y) \boldsymbol\cdot z$.
Therefore, we see that an internal group will have type ${}^*\mathcal{G} = \mathcal{G} = ((1,0),(\iota,1),(\boldsymbol\cdot,2))$ and will ${}^*$-satisfy the following equations:
$x \boldsymbol\cdot 1 \approx 1 \boldsymbol\cdot x \approx x$
$x \boldsymbol\cdot \iota(x) \approx \iota(x) \boldsymbol\cdot x \approx 1$
$x\boldsymbol\cdot(y\boldsymbol\cdot z) \approx (x \boldsymbol\cdot y) \boldsymbol\cdot z$.
Consequently, an internal group is a group since ${}^*$-satisfy and satisfy mean the same for these equations.
\end{example}

A similar thing can be done for rings.

\begin{example}\label{}
    Vector spaces are a little different depending on if the field of scalars $\F$ is finite or infinite. In either case, vector space $\V$ can be described as an ordered tuple
    \[
    \V = (V, \mathcal{O}_\V) \text{ for } \mathcal{O}_\V = (0_\V,\iota_\V, ({s_r}_\V|r \in F) ,+_\V),
    \]
    where ${s_r}_\V$ is the scalar operator associated with the element $r \in F$. We see that $\V$ satisfies the (potentially infinite) set of equations: 
    \begin{enumerate}
        \item   $x_1 +_\V (x_2 +_\V x_3) \approx (x_1 +_\V x_2) +_\V x_3$.
        \item   $x_1 +_\V x_2 \approx x_2 +_\V x_1$.
        \item   $x_1 +_\V 0_\V \approx x_1$.
        \item   $x_1 +_\V \iota(x_1) \approx 0_\V$.
        \item   ${s_r}_\V(x_1 +_\V x_2) \approx {s_r}_\V(x_1) +_\V {s_r}_\V(x_2)$.
        \item   Given that $r \times_\F t = a$, we have the equation ${s_r}_\V({s_t}_\V(x_1)) \approx {s_a}_\V(x_1)$.
        \item   Given that $1$ is the multiplicative identity of $\F$, we have the equation ${s_1}_\V(x_1) \approx x_1$.
        \item   Given that $r +_\F t = b$, we have the equation ${s_r}_\V(x_1) +_\V {s_t}_\V(x_1) \approx {s_b}_\V(x_1)$.
    \end{enumerate}
    The first four of these equations establish the abelian group properties of the vectors, and the final four equations show the ring structure of the scalars.

    When transferred, we see that ${}^*\V$ satisfies the (potentially hyperinfinite) set of equations (for $r,t \in {}^*F$):
    \begin{enumerate}
        \item   $x_1 +_\V (x_2 +_\V x_3) \approx (x_1 +_\V x_2) +_\V x_3$.
        \item   $x_1 +_\V x_2 \approx x_2 +_\V x_1$.
        \item   $x_1 +_\V 0_\V \approx x_1$.
        \item   $x_1 +_\V \iota(x_1) \approx 0_\V$.
        \item   ${s_r}_\V(x_1 +_\V x_2) \approx {s_r}_\V(x_1) +_\V {s_r}_\V(x_2)$.
        \item   Given that $r {}^*\times_\F t = a$, we have the equation ${s_r}_\V({s_t}_\V(x_1)) \approx {s_a}_\V(x_1)$.
        \item   Given that $1$ is the internal multiplicative identity of $\F$, we have the equation ${s_1}_\V(x_1) \approx x_1$.
        \item   Given that $r {}^*+_\F t = b$, we have the equation ${s_r}_\V(x_1) +_\V {s_t}_\V(x_1) \approx {s_b}_\V(x_1)$.
    \end{enumerate}
    Since we know that for any $r,t \in F$, we have that $r \times_\F t = r {}^*\times_\F t$ and $r \times_\F t = r {}^*+_\F t$, that $1$ will be the internal multiplicative identity, and that ${}^*\V$ satisfies a superset of the defining equations of a vector space, then we know ${}^*\V$ is a member of the variety of vector spaces.
\end{example}

This neat agreement of satisfies and ${}^*$-satisfies breaks down somewhat for algebras with infinitary operations

\begin{example}\label{}
    Let $\B$ be a countably complete boolean algebra and define $\B' = (B: 0_\B,1_\B,c_\B,\wedge_{\B'},\vee_{\B'})$, where $\wedge_{\B'}$ and $\vee_{\B'}$ are the $\omega_0$-ary meet and join operations. We see that $\B'$ satisfies the following equations:
    \begin{enumerate}
        \item $c_\B(0_\B) \approx 1_\B$.
        \item $c_\B(c_\B(x_1)) \approx x_1$.
        \item $\wedge_{\B'}(x_1,x_1,x_1,...) \approx x_1$.
        \item $\vee_{\B'}(x_1,x_1,x_1,...) \approx x_1$.
        \item $\wedge_{\B'}(1_\B,x_2,x_3,...) \approx \wedge_{\B'}(x_2,x_3,...)$.
        \item $\vee_{\B'}(0_\B,x_2,x_3,...) \approx \vee_{\B'}(x_2,x_3,...)$.
        \item $\wedge_{\B'}(0_\B,x_2,x_3,...) \approx 0_\B$.
        \item $\vee_{\B'}(1_\B,x_2,x_3,...) \approx 1_\B$.
        \item $\wedge_{\B'}(x_1,c_\B(x_1),x_3,...) \approx 0_\B$.
        \item $\vee_{\B'}(x_1,c_\B(x_1),x_3,...) \approx 1_\B$.
        \item $\wedge_{\B'}(x_1,\vee_{\B'}(x_1,x_{1,2}...),\vee_{\B'}(x_1,x_{2,2},...),...) \approx x_1$.
        \item $\vee_{\B'}(x_1,\wedge_{\B'}(x_1,x_{1,2}...),\wedge_{\B'}(x_1,x_{2,2},...),...) \approx x_1$.
        \item $\wedge_{\B'}(\vee_{\B'}(x_{1,1},x_{1,2},...), \vee_{\B'}(x_{2,1},x_{2,2},...), ...) \approx \vee_{\B'}(\wedge_{\B'}(x_{1,1},x_{1,2},...), \wedge_{\B'}(x_{2,1},x_{2,2},...), ...)$.
        \item The infinite set of equations that rearrange the order of the arguments in $\wedge_{\B'}$. In other words, for every permutation $p$ of the natural numbers, equation $\wedge_{\B'}(x_{p(1)},x_{p(2)},...) \approx \wedge_{\B'}(x_1,x_2,...)$.
        \item The infinite set of equations that rearrange the order of the arguments in $\vee_{\B'}$. In other words, for every permutation $p$ of the natural numbers, equation $\vee_{\B'}(x_{p(1)},x_{p(2)},...) \approx \vee_{\B'}(x_1,x_2,...)$.
    \end{enumerate}
    We observe that every equation except the first two have infinitely many symbols and, hence, will not be identical to their ${}^*$-mappings. We see that ${}^*\B'$ ${}^*$-satisfies the following equations, where ${}^*\wedge_{\B'}$ and ${}^*\vee_{\B'}$ are ${}^*\omega_0$-ary instead of $\omega_0$-ary:
    \begin{enumerate}
        \item $c_\B(0_\B) \approx 1_\B$.
        \item $c_\B(c_\B(x_1)) \approx x_1$.
        \item ${}^*\wedge_{\B'}(x_1,x_1,x_1,...) \approx x_1$.
        \item ${}^*\vee_{\B'}(x_1,x_1,x_1,...) \approx x_1$.
        \item ${}^*\wedge_{\B'}(1_\B,x_2,x_3,...) \approx {}^*\wedge_{\B'}(x_2,x_3,...)$.
        \item ${}^*\vee_{\B'}(0_\B,x_2,x_3,...) \approx {}^*\vee_{\B'}(x_2,x_3,...)$.
        \item ${}^*\wedge_{\B'}(0_\B,x_2,x_3,...) \approx 0_\B$.
        \item ${}^*\vee_{\B'}(1_\B,x_2,x_3,...) \approx 1_\B$.
        \item ${}^*\wedge_{\B'}(x_1,c_\B(x_1),x_3,...) \approx 0_\B$.
        \item ${}^*\vee_{\B'}(x_1,c_\B(x_1),x_3,...) \approx 1_\B$.
        \item ${}^*\wedge_{\B'}(x_1,{}^*\vee_{\B'}(x_1,x_{1,2}...),{}^*\vee_{\B'}(x_1,x_{2,2},...),...) \approx x_1$.
        \item ${}^*\vee_{\B'}(x_1,{}^*\wedge_{\B'}(x_1,x_{1,2}...),{}^*\wedge_{\B'}(x_1,x_{2,2},...),...) \approx x_1$.
        \item ${}^*\wedge_{\B'}({}^*\vee_{\B'}(x_{1,1},x_{1,2},...), {}^*\vee_{\B'}(x_{2,1},x_{2,2},...), ...) \approx {}^*\vee_{\B'}({}^*\wedge_{\B'}(x_{1,1},x_{1,2},...), {}^*\wedge_{\B'}(x_{2,1},x_{2,2},...), ...)$.
        \item The hyperinfinite set of equations that rearrange the order of the arguments in ${}^*\wedge_{\B'}$. In other words, for every permutation $p$ of the nonstandard natural numbers, equation ${}^*\wedge_{\B'}(x_{p(1)},x_{p(2)},...) \approx {}^*\wedge_{\B'}(x_1,x_2,...)$.
        \item The hyperinfinite set of equations that rearrange the order of the arguments in ${}^*\vee_{\B'}$. In other word,s for every permutation $p$ of the nonstandard natural numbers, equation ${}^*\vee_{\B'}(x_{p(1)},x_{p(2)},...) \approx {}^*\vee_{\B'}(x_1,x_2,...)$.
    \end{enumerate}
    Algebra ${}^*\B'$ is not in itself an $\omega_0$-ary boolean algbra, nor does it satisfy the external equations using only the symbols in its language; however, this will be covered in much more depth in a later paper.
\end{example}

Occasionally, we will find it useful to apply the overspill principle of nonstandard methods, much like the way it is done in nonstandard analysis. A simple example of the overspill principle is that there is a nonstandard natural number that is divisible by every positive natural number. This is done by applying overspill to the fact that for each positive natural number, there exists a natural number that is divisible by all positive natural numbers less than or equal to it. The overspill principle states that if a property holds for all sufficiently large natural numbers, it will hold for some nonstandard hyperfinite number. 

The concurrency principle is equivalent to overspill, so we will take the time to define concurrency now. Relation $r$ with domain $A \in V(U)$ and codomain $B \in V(U)$ is \textbf{concurrent} if and only if for each finite subset $C \subseteq A$ there exists $r$-bound $b \in B$ such that $\forall \text{ } x \in C$, we have $x r b$. If we define a superstructure enlargement of $V(U)$, the concurrency principle holds for enlargement ${}^*V(U)$, which states that if $r \in V(U)$ is a concurrent relation with domain $A$ and codomain $B$, then there exists $b \in {}^*B$ such that for all $x \in A$, we have ${}^*x {}^*r b$.

In addition, the concept of a clone of an algebra will be touched on, so we will give a brief rundown here. For more detail, see \cite{Hobby}. A clone of algebra $\A$ with operations in $\mathcal{O}$, denoted $Clo \A$, is the collection of operations that contains every possible projection operation and is closed under finite compositions. We denote $Clo_n\A$ to be the collection of operations from $Clo\A$ that have arity $n$.

In the later portions of this paper, we will discuss direct products of algebras in relation to our nonstandard methods. Because we assume that all our standard algebras have urelements as the elements of their underlying sets, we will at times refrain from using the typical cartesian product when notating and computing direct products, instead using the category theoretical approach, which is isomorphic. 

For example, assume we are given a two-element boolean algebra, $\B = (\{b_1,b_2\}; 0_\B, 1_\B, \tilde{}_\B, \wedge_\B, \vee_\B)$, and a four-element boolean algebra, $\O = (\{o_1,o_2,o_3,o_4\}; 0_\O, 1_\O, \tilde{}_\O, \wedge_\O, \vee_\O)$, such that $0_\B() = b_1, \tilde{b_1}_\B = b_2,, \tilde{o_1}_\O = o_4$, and $ \tilde{o_2}_\O = o_3$. Also assume we are given the following:
    \begin{center}
    \hfill
    \renewcommand\arraystretch{1.3}
    \setlength\doublerulesep{0pt}
    \begin{tabular}{r|| c | c }
    $\wedge_\B$ & $b_1$ & $b_2$ \\
    \hline\hline
    $b_1$ & $b_1$ & $b_1$  \\ 
    \hline
    $b_2$ & $b_1$ & $b_2$ \\ 
    \hline
    \end{tabular}
    \hfill and \hfill
    \renewcommand\arraystretch{1.3}
    \setlength\doublerulesep{0pt}
    \begin{tabular}{r|| c | c }
    $\vee_\B$ & $b_1$ & $b_2$ \\
    \hline\hline
    $b_1$ & $b_1$ & $b_2$  \\ 
    \hline
    $b_2$ & $b_2$ & $b_2$ \\ 
    \hline
    \end{tabular}
    \hfill\phantom.
    \end{center}

    \begin{center}
    \hfill
    \renewcommand\arraystretch{1.3}
    \setlength\doublerulesep{0pt}
    \begin{tabular}{r|| c | c | c | c}
    $\wedge_\O$ & $o_1$ & $o_2$ & $o_3$ & $o_4$ \\
    \hline\hline
    $o_1$ & $o_1$ & $o_1$ & $o_1$ & $o_1$ \\ 
    \hline
    $o_2$ & $o_1$ & $o_2$ & $o_1$ & $o_2$ \\ 
    \hline
    $o_3$ & $o_1$ & $o_1$ & $o_3$ & $o_3$ \\ 
    \hline
    $o_4$ & $o_1$ & $o_2$ & $o_3$ & $o_4$ \\
    \hline
    \end{tabular}
 \hfill and \hfill
    \renewcommand\arraystretch{1.3}
    \setlength\doublerulesep{0pt}
    \begin{tabular}{r|| c | c | c | c}
    $\vee_\O$ & $o_1$ & $o_2$ & $o_3$ & $o_4$ \\
    \hline\hline
    $o_1$ & $o_1$ & $o_2$ & $o_3$ & $o_4$ \\ 
    \hline
    $o_2$ & $o_2$ & $o_2$ & $o_4$ & $o_4$ \\ 
    \hline
    $o_3$ & $o_3$ & $o_4$ & $o_3$ & $o_4$ \\ 
    \hline
    $o_4$ & $o_4$ & $o_4$ & $o_4$ & $o_4$ \\
    \hline
    \end{tabular}
    \hfill\phantom.
    \end{center}

We will construct direct product $\P$ to be some eight-element set. With care, we can use any eight urelements $\{s,t,u,v,w,x,y,z\}$ whose projections $\rho_\B,\rho_\O$ behave as follows: 

    \begin{center}
    \hfill
    \renewcommand\arraystretch{1.3}
    \setlength\doublerulesep{0pt}
    \begin{tabular}{r|| c }
    $\xi$ & $\rho_\B(\xi)$ \\
    \hline\hline
    $s$ & $b_1$ \\ 
    \hline
    $t$ & $b_1$ \\ 
    \hline
    $u$ & $b_1$ \\ 
    \hline
    $v$ & $b_1$ \\
    \hline
    $w$ & $b_2$ \\ 
    \hline
    $x$ & $b_2$ \\ 
    \hline
    $y$ & $b_2$ \\ 
    \hline
    $x$ & $b_2$ \\
    \hline
    \end{tabular}
 \hfill and \hfill
    \renewcommand\arraystretch{1.3}
    \setlength\doublerulesep{0pt}
    \begin{tabular}{r|| c }
    $\xi$ & $\rho_\O(\xi)$ \\
    \hline\hline
    $s$ & $o_1$ \\ 
    \hline
    $t$ & $o_2$ \\ 
    \hline
    $u$ & $o_3$ \\ 
    \hline
    $v$ & $o_4$ \\
    \hline
    $w$ & $o_1$ \\ 
    \hline
    $x$ & $o_2$ \\ 
    \hline
    $y$ & $o_3$ \\ 
    \hline
    $x$ & $o_4$ \\
    \hline
    \end{tabular}
    \hfill\phantom.
    \end{center}

This algebra will have associated binary operations defined as follows:

    \begin{center}
    \renewcommand\arraystretch{1.3}
    \setlength\doublerulesep{0pt}
    \begin{tabular}{r|| c | c | c | c | c | c | c | c}
    $\wedge_\P$ & $s$ & $t$ & $u$ & $v$ & $w$ & $x$ & $y$ & $z$ \\
    \hline\hline
    $s$ & $s$ & $s$ & $s$ & $s$ & $s$ & $s$ & $s$ & $s$ \\ 
    \hline
    $t$ & $s$ & $t$ & $s$ & $t$ & $s$ & $t$ & $s$ & $t$ \\ 
    \hline
    $u$ & $s$ & $s$ & $u$ & $u$ & $s$ & $s$ & $u$ & $u$ \\ 
    \hline
    $v$ & $s$ & $t$ & $u$ & $v$ & $s$ & $t$ & $u$ & $v$ \\
    \hline
    $w$ & $s$ & $s$ & $s$ & $s$ & $w$ & $w$ & $w$ & $w$ \\ 
    \hline
    $x$ & $s$ & $t$ & $s$ & $t$ & $w$ & $x$ & $w$ & $x$ \\ 
    \hline
    $y$ & $s$ & $s$ & $u$ & $u$ & $w$ & $w$ & $y$ & $y$ \\ 
    \hline
    $z$ & $s$ & $t$ & $u$ & $v$ & $w$ & $x$ & $y$ & $z$ \\
    \hline
    \end{tabular}
    \end{center}
    
and

    \begin{center}
    \renewcommand\arraystretch{1.3}
    \setlength\doublerulesep{0pt}
    \begin{tabular}{r|| c | c | c | c | c | c | c | c}
    $\vee_\P$ & $s$ & $t$ & $u$ & $v$ & $w$ & $x$ & $y$ & $z$ \\
    \hline\hline
    $s$ & $s$ & $t$ & $u$ & $v$ & $w$ & $x$ & $y$ & $z$ \\ 
    \hline
    $t$ & $t$ & $t$ & $v$ & $v$ & $x$ & $x$ & $z$ & $z$ \\ 
    \hline
    $u$ & $u$ & $v$ & $u$ & $v$ & $y$ & $z$ & $y$ & $z$ \\ 
    \hline
    $v$ & $v$ & $v$ & $v$ & $v$ & $z$ & $z$ & $z$ & $z$ \\
    \hline
    $w$ & $w$ & $x$ & $y$ & $z$ & $w$ & $x$ & $y$ & $z$ \\ 
    \hline
    $x$ & $x$ & $x$ & $z$ & $z$ & $x$ & $x$ & $z$ & $z$ \\ 
    \hline
    $y$ & $y$ & $z$ & $y$ & $z$ & $y$ & $z$ & $y$ & $z$ \\ 
    \hline
    $z$ & $z$ & $z$ & $z$ & $z$ & $z$ & $z$ & $z$ & $z$ \\
    \hline
    \end{tabular}
    \end{center}

Now to show that $\P = \Pi\{\B,\O\}$, which is the categorical direct product of $\B$ and $\O$, we must show that the following commutative diagram commutes for any boolean algebra $\A$ with homormorphisms $\beta_\B$ and $\beta_\O$ to $\B$ and $\O$, respectively.

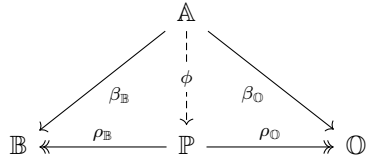
\[\begin{tikzcd}
	&& \A \\
	\\
	\B && \P && \O
	\arrow["{\beta_\B}", from=1-3, to=3-1]
	\arrow["\phi"{description}, dashed, from=1-3, to=3-3]
	\arrow["{\beta_\O}"', from=1-3, to=3-5]
	\arrow["{\rho_\B}"', two heads, from=3-3, to=3-1]
	\arrow["{\rho_\O}", two heads, from=3-3, to=3-5]
\end{tikzcd}\]

To do this we must show that $\beta_\B(\xi) = \rho_\B(\phi(\xi))$ and $\beta_\O(\xi) = \rho_\O(\phi(\xi))$ for all $\xi \in A$. Let $a_1,a_2 \in A$. Assuming that $\beta_\B(a_1) = b_1$ and $\beta_\O(a_1) = o_1$, it is clear that $\phi(a_1)$ must be $s$. By a similar argument, we can fill out the following table:

    \begin{center}
    \renewcommand\arraystretch{1.3}
    \setlength\doublerulesep{0pt}
    \begin{tabular}{c | c || c }
    $\beta_\B(\xi)$ & $\beta_\O(\xi)$ & $\phi(\xi)$ \\
    \hline\hline
    $b_1$ & $o_1$ & $s$ \\ 
    \hline
    $b_1$ & $o_2$ & $t$ \\ 
    \hline
    $b_1$ & $o_3$ & $u$ \\ 
    \hline
    $b_1$ & $o_4$ & $v$ \\
    \hline
    $b_2$ & $o_1$ & $w$ \\ 
    \hline
    $b_2$ & $o_2$ & $x$ \\ 
    \hline
    $b_2$ & $o_3$ & $y$ \\ 
    \hline
    $b_2$ & $o_4$ & $z$ \\
    \hline
    \end{tabular}
    \end{center}

Now we must show that $\phi$ preserves the operations and is, thus, a homomorphism. For $a_1 \wedge a_2$, we know that $\beta_\B(a_1 \wedge a_2) = \beta_\B(a_1) \wedge \beta_\B(a_2)$ and $\beta_\O(a_1 \wedge a_2) = \beta_\O(a_1) \wedge \beta_\O(a_2)$, and we see that the following tables show the preservation of operations by $\phi$.

First for meets:
    \begin{center}
    \renewcommand\arraystretch{1.3}
    \setlength\doublerulesep{0pt}
    \begin{longtable}{c | c | c | c || c | c}
    $\beta_\B(a_1)$ & $\beta_\O(a_1)$ & $\beta_\B(a_2)$ & $\beta_\O(a_2)$ & $\phi(a_1 \wedge a_2)$ & $\phi(a_1) \wedge\phi(a_2)$ \\
    \hline\hline
    $b_1$ & $o_1$ & $b_1$ & $o_1$ & $s$ & $s$ \\ 
    \hline
    $b_1$ & $o_2$ & $b_1$ & $o_1$ & $s$ & $s$ \\ 
    \hline
    $b_1$ & $o_3$ & $b_1$ & $o_1$ & $s$ & $s$ \\ 
    \hline
    $b_1$ & $o_4$ & $b_1$ & $o_1$ & $s$ & $s$ \\ 
    \hline
    $b_2$ & $o_1$ & $b_1$ & $o_1$ & $s$ & $s$ \\ 
    \hline
    $b_2$ & $o_2$ & $b_1$ & $o_1$ & $s$ & $s$ \\ 
    \hline
    $b_2$ & $o_3$ & $b_1$ & $o_1$ & $s$ & $s$ \\ 
    \hline
    $b_2$ & $o_4$ & $b_1$ & $o_1$ & $s$ & $s$ \\ 
    \hline
    $b_1$ & $o_1$ & $b_1$ & $o_2$ & $s$ & $s$ \\ 
    \hline
    $b_1$ & $o_2$ & $b_1$ & $o_2$ & $t$ & $t$ \\ 
    \hline
    $b_1$ & $o_3$ & $b_1$ & $o_2$ & $s$ & $s$ \\ 
    \hline
    $b_1$ & $o_4$ & $b_1$ & $o_2$ & $t$ & $t$ \\ 
    \hline
    $b_2$ & $o_1$ & $b_1$ & $o_2$ & $s$ & $s$ \\ 
    \hline
    $b_2$ & $o_2$ & $b_1$ & $o_2$ & $t$ & $t$ \\ 
    \hline
    $b_2$ & $o_3$ & $b_1$ & $o_2$ & $s$ & $s$ \\ 
    \hline
    $b_2$ & $o_4$ & $b_1$ & $o_2$ & $t$ & $t$ \\
    \hline
    $b_1$ & $o_1$ & $b_1$ & $o_3$ & $s$ & $s$ \\ 
    \hline
    $b_1$ & $o_2$ & $b_1$ & $o_3$ & $s$ & $s$ \\ 
    \hline
    $b_1$ & $o_3$ & $b_1$ & $o_3$ & $u$ & $u$ \\ 
    \hline
    $b_1$ & $o_4$ & $b_1$ & $o_3$ & $u$ & $u$ \\
    \hline
    $b_2$ & $o_1$ & $b_1$ & $o_3$ & $s$ & $s$ \\ 
    \hline
    $b_2$ & $o_2$ & $b_1$ & $o_3$ & $s$ & $s$ \\ 
    \hline
    $b_2$ & $o_3$ & $b_1$ & $o_3$ & $u$ & $u$ \\ 
    \hline
    $b_2$ & $o_4$ & $b_1$ & $o_3$ & $u$ & $u$ \\
    \hline
    $b_1$ & $o_1$ & $b_1$ & $o_4$ & $s$ & $s$ \\ 
    \hline
    $b_1$ & $o_2$ & $b_1$ & $o_4$ & $t$ & $t$ \\ 
    \hline
    $b_1$ & $o_3$ & $b_1$ & $o_4$ & $u$ & $u$ \\ 
    \hline
    $b_1$ & $o_4$ & $b_1$ & $o_4$ & $v$ & $v$ \\
    \hline
    $b_2$ & $o_1$ & $b_1$ & $o_4$ & $s$ & $s$ \\ 
    \hline
    $b_2$ & $o_2$ & $b_1$ & $o_4$ & $t$ & $t$ \\ 
    \hline
    $b_2$ & $o_3$ & $b_1$ & $o_4$ & $u$ & $u$ \\ 
    \hline
    $b_2$ & $o_4$ & $b_1$ & $o_4$ & $v$ & $v$ \\
    \hline
    $b_1$ & $o_1$ & $b_2$ & $o_1$ & $s$ & $s$ \\ 
    \hline
    $b_1$ & $o_2$ & $b_2$ & $o_1$ & $s$ & $s$ \\ 
    \hline
    $b_1$ & $o_3$ & $b_2$ & $o_1$ & $s$ & $s$ \\ 
    \hline
    $b_1$ & $o_4$ & $b_2$ & $o_1$ & $s$ & $s$ \\
    \hline
    $b_2$ & $o_1$ & $b_2$ & $o_1$ & $w$ & $w$ \\ 
    \hline
    $b_2$ & $o_2$ & $b_2$ & $o_1$ & $w$ & $w$ \\ 
    \hline
    $b_2$ & $o_3$ & $b_2$ & $o_1$ & $w$ & $w$ \\ 
    \hline
    $b_2$ & $o_4$ & $b_2$ & $o_1$ & $w$ & $w$ \\
    \hline
    $b_1$ & $o_1$ & $b_2$ & $o_2$ & $s$ & $s$ \\ 
    \hline
    $b_1$ & $o_2$ & $b_2$ & $o_2$ & $t$ & $t$ \\ 
    \hline
    $b_1$ & $o_3$ & $b_2$ & $o_2$ & $s$ & $s$ \\ 
    \hline
    $b_1$ & $o_4$ & $b_2$ & $o_2$ & $t$ & $t$ \\
    \hline
    $b_2$ & $o_1$ & $b_2$ & $o_2$ & $w$ & $w$ \\ 
    \hline
    $b_2$ & $o_2$ & $b_2$ & $o_2$ & $x$ & $x$ \\ 
    \hline
    $b_2$ & $o_3$ & $b_2$ & $o_2$ & $w$ & $w$ \\ 
    \hline
    $b_2$ & $o_4$ & $b_2$ & $o_2$ & $x$ & $x$ \\
    \hline
    $b_1$ & $o_1$ & $b_2$ & $o_3$ & $s$ & $s$ \\ 
    \hline
    $b_1$ & $o_2$ & $b_2$ & $o_3$ & $s$ & $s$ \\ 
    \hline
    $b_1$ & $o_3$ & $b_2$ & $o_3$ & $u$ & $u$ \\ 
    \hline
    $b_1$ & $o_4$ & $b_2$ & $o_3$ & $u$ & $u$ \\
    \hline
    $b_2$ & $o_1$ & $b_2$ & $o_3$ & $w$ & $w$ \\ 
    \hline
    $b_2$ & $o_2$ & $b_2$ & $o_3$ & $w$ & $w$ \\ 
    \hline
    $b_2$ & $o_3$ & $b_2$ & $o_3$ & $y$ & $y$ \\ 
    \hline
    $b_2$ & $o_4$ & $b_2$ & $o_3$ & $y$ & $y$ \\
    \hline
    $b_1$ & $o_1$ & $b_2$ & $o_4$ & $s$ & $s$ \\ 
    \hline
    $b_1$ & $o_2$ & $b_2$ & $o_4$ & $t$ & $t$ \\ 
    \hline
    $b_1$ & $o_3$ & $b_2$ & $o_4$ & $u$ & $u$ \\ 
    \hline
    $b_1$ & $o_4$ & $b_2$ & $o_4$ & $v$ & $v$ \\
    \hline
    $b_2$ & $o_1$ & $b_2$ & $o_4$ & $w$ & $w$ \\ 
    \hline
    $b_2$ & $o_2$ & $b_2$ & $o_4$ & $x$ & $x$ \\ 
    \hline
    $b_2$ & $o_3$ & $b_2$ & $o_4$ & $y$ & $y$ \\ 
    \hline
    $b_2$ & $o_4$ & $b_2$ & $o_4$ & $z$ & $z$ \\
    \hline
    \end{longtable}
    \end{center}

And for joins as well:
    \begin{center}
    \renewcommand\arraystretch{1.3}
    \setlength\doublerulesep{0pt}
    \begin{longtable}{c | c | c | c || c | c}
    $\beta_\B(a_1)$ & $\beta_\O(a_1)$ & $\beta_\B(a_2)$ & $\beta_\O(a_2)$ & $\phi(a_1 \vee a_2)$ & $\phi(a_1) \vee\phi(a_2)$ \\
    \hline\hline
    $b_1$ & $o_1$ & $b_1$ & $o_1$ & $s$ & $s$ \\ 
    \hline
    $b_1$ & $o_2$ & $b_1$ & $o_1$ & $t$ & $t$ \\ 
    \hline
    $b_1$ & $o_3$ & $b_1$ & $o_1$ & $u$ & $u$ \\ 
    \hline
    $b_1$ & $o_4$ & $b_1$ & $o_1$ & $v$ & $v$ \\
    \hline
    $b_2$ & $o_1$ & $b_1$ & $o_1$ & $w$ & $w$ \\ 
    \hline
    $b_2$ & $o_2$ & $b_1$ & $o_1$ & $x$ & $x$ \\ 
    \hline
    $b_2$ & $o_3$ & $b_1$ & $o_1$ & $y$ & $y$ \\ 
    \hline
    $b_2$ & $o_4$ & $b_1$ & $o_1$ & $z$ & $z$ \\
    \hline
    $b_1$ & $o_1$ & $b_1$ & $o_2$ & $t$ & $t$ \\ 
    \hline
    $b_1$ & $o_2$ & $b_1$ & $o_2$ & $t$ & $t$ \\ 
    \hline
    $b_1$ & $o_3$ & $b_1$ & $o_2$ & $s$ & $s$ \\ 
    \hline
    $b_1$ & $o_4$ & $b_1$ & $o_2$ & $v$ & $v$ \\
    \hline
    $b_2$ & $o_1$ & $b_1$ & $o_2$ & $x$ & $x$ \\ 
    \hline
    $b_2$ & $o_2$ & $b_1$ & $o_2$ & $x$ & $x$ \\ 
    \hline
    $b_2$ & $o_3$ & $b_1$ & $o_2$ & $z$ & $z$ \\ 
    \hline
    $b_2$ & $o_4$ & $b_1$ & $o_2$ & $z$ & $z$ \\
    \hline
    $b_1$ & $o_1$ & $b_1$ & $o_3$ & $u$ & $u$ \\ 
    \hline
    $b_1$ & $o_2$ & $b_1$ & $o_3$ & $v$ & $v$ \\ 
    \hline
    $b_1$ & $o_3$ & $b_1$ & $o_3$ & $u$ & $u$ \\ 
    \hline
    $b_1$ & $o_4$ & $b_1$ & $o_3$ & $v$ & $v$ \\
    \hline
    $b_2$ & $o_1$ & $b_1$ & $o_3$ & $y$ & $y$ \\ 
    \hline
    $b_2$ & $o_2$ & $b_1$ & $o_3$ & $z$ & $z$ \\ 
    \hline
    $b_2$ & $o_3$ & $b_1$ & $o_3$ & $y$ & $y$ \\ 
    \hline
    $b_2$ & $o_4$ & $b_1$ & $o_3$ & $z$ & $z$ \\
    \hline
    $b_1$ & $o_1$ & $b_1$ & $o_4$ & $v$ & $v$ \\ 
    \hline
    $b_1$ & $o_2$ & $b_1$ & $o_4$ & $v$ & $v$ \\ 
    \hline
    $b_1$ & $o_3$ & $b_1$ & $o_4$ & $v$ & $v$ \\ 
    \hline
    $b_1$ & $o_4$ & $b_1$ & $o_4$ & $v$ & $v$ \\
    \hline
    $b_2$ & $o_1$ & $b_1$ & $o_4$ & $z$ & $z$ \\ 
    \hline
    $b_2$ & $o_2$ & $b_1$ & $o_4$ & $z$ & $z$ \\ 
    \hline
    $b_2$ & $o_3$ & $b_1$ & $o_4$ & $z$ & $z$ \\ 
    \hline
    $b_2$ & $o_4$ & $b_1$ & $o_4$ & $z$ & $z$ \\
    \hline
    $b_1$ & $o_1$ & $b_2$ & $o_1$ & $w$ & $w$ \\ 
    \hline
    $b_1$ & $o_2$ & $b_2$ & $o_1$ & $x$ & $x$ \\ 
    \hline
    $b_1$ & $o_3$ & $b_2$ & $o_1$ & $y$ & $y$ \\ 
    \hline
    $b_1$ & $o_4$ & $b_2$ & $o_1$ & $z$ & $z$ \\
    \hline
    $b_2$ & $o_1$ & $b_2$ & $o_1$ & $w$ & $w$ \\ 
    \hline
    $b_2$ & $o_2$ & $b_2$ & $o_1$ & $x$ & $x$ \\ 
    \hline
    $b_2$ & $o_3$ & $b_2$ & $o_1$ & $y$ & $y$ \\ 
    \hline
    $b_2$ & $o_4$ & $b_2$ & $o_1$ & $z$ & $z$ \\
    \hline
    $b_1$ & $o_1$ & $b_2$ & $o_2$ & $x$ & $x$ \\ 
    \hline
    $b_1$ & $o_2$ & $b_2$ & $o_2$ & $x$ & $x$ \\ 
    \hline
    $b_1$ & $o_3$ & $b_2$ & $o_2$ & $z$ & $z$ \\ 
    \hline
    $b_1$ & $o_4$ & $b_2$ & $o_2$ & $z$ & $z$ \\
    \hline
    $b_2$ & $o_1$ & $b_2$ & $o_2$ & $x$ & $x$ \\ 
    \hline
    $b_2$ & $o_2$ & $b_2$ & $o_2$ & $x$ & $x$ \\ 
    \hline
    $b_2$ & $o_3$ & $b_2$ & $o_2$ & $z$ & $z$ \\ 
    \hline
    $b_2$ & $o_4$ & $b_2$ & $o_2$ & $z$ & $z$ \\
    \hline
    $b_1$ & $o_1$ & $b_2$ & $o_3$ & $y$ & $y$ \\ 
    \hline
    $b_1$ & $o_2$ & $b_2$ & $o_3$ & $z$ & $z$ \\ 
    \hline
    $b_1$ & $o_3$ & $b_2$ & $o_3$ & $y$ & $y$ \\ 
    \hline
    $b_1$ & $o_4$ & $b_2$ & $o_3$ & $z$ & $z$ \\
    \hline
    $b_2$ & $o_1$ & $b_2$ & $o_3$ & $y$ & $y$ \\ 
    \hline
    $b_2$ & $o_2$ & $b_2$ & $o_3$ & $z$ & $z$ \\ 
    \hline
    $b_2$ & $o_3$ & $b_2$ & $o_3$ & $y$ & $y$ \\ 
    \hline
    $b_2$ & $o_4$ & $b_2$ & $o_3$ & $z$ & $z$ \\
    \hline
    $b_1$ & $o_1$ & $b_2$ & $o_4$ & $z$ & $z$ \\ 
    \hline
    $b_1$ & $o_2$ & $b_2$ & $o_4$ & $z$ & $z$ \\ 
    \hline
    $b_1$ & $o_3$ & $b_2$ & $o_4$ & $z$ & $z$ \\ 
    \hline
    $b_1$ & $o_4$ & $b_2$ & $o_4$ & $z$ & $z$ \\
    \hline
    $b_2$ & $o_1$ & $b_2$ & $o_4$ & $z$ & $z$ \\ 
    \hline
    $b_2$ & $o_2$ & $b_2$ & $o_4$ & $z$ & $z$ \\ 
    \hline
    $b_2$ & $o_3$ & $b_2$ & $o_4$ & $z$ & $z$ \\ 
    \hline
    $b_2$ & $o_4$ & $b_2$ & $o_4$ & $z$ & $z$ \\
    \hline
    \end{longtable}
    \end{center}

The following theorem is a well-known result in model theory (see \cite[p.103,ex.22-24]{Enderton}), which we will use in \ref{UrProdMonad} to construct isomorphic copies of direct products with underlying sets being sets of urelements instead of performing similar processes to our previous tedious example.

\begin{theorem}
    Let $J$ be an indexing set, and for each index $j \in J$, we take algebra $\A_j$. Let $\P$ denote the direct product $\Pi\{\A_j|j \in J\}$ and let $\sigma: P \rightarrow U$, where $U$ is a set of urelements of cardinality at least that of $P$\textbf{---}that is, $\Pi\{|A_j||j \in J\} \leq |U|$\textbf{---}and assume that $\sigma$ is injective. Then there is a unique algebraic structure $\V$ whose underlying set is $V = \sigma[P]$ such that $\sigma:\P \rightarrow \V$ is an isomorphism.
\end{theorem}
\begin{proof}
    For each $o \in \mathcal{O}$, we define operation $o_\V$ on $V$ by the formula $o_\V(v_1,v_2,...) = \sigma(o_\P(\sigma^{-1}(v_1),\sigma^{-1}(v_2),...)$. This is well defined as $\sigma$ is bijective between its domain and range, and it clearly allows $\sigma$ to satisfy the homomorphism property for each operation of $\P$, thus showing that $\P \cong \V$.
\end{proof}

\subsection{Cardinal Arities vs Ordinal Arities}
Before moving into the main body of this work, where we use ordinal arities for our operations, we will briefly discuss the distinguishing features between ordinal arities and cardinal arities. We will also look at how this motivates us to use ordinal arities in our initial investigation.

The primary difference between ordinal and cardinal arities is how the entries are indexed, in which the operands, the arguments of the operation, are put into to determine the output of the function. An operation of algebra $\A$ that has finite cardinal arity n is a function from $\{\{a_1, a_2, ..., a_n\}| a_1, a_2, ..., a_n \in A\}$ to $A$, which allows for any size of sets from 1 to n of elements of $A$. This, of course, is not how operations are typically considered as there is no method of referring to specific entries. This is fine for commutative operations, but noncommutative operations will not work under such a language. An operation defined with ordinal arity $\eta$, instead, is a function from $A^\eta$ to $A$, allowing for specific entries in the operation to correspond with those entries in the product, providing us more ability to define operations.

The ordering of the operations of an algebra can also benefit from assuming an ordinal index as opposed to a cardinal value. Without the ability to index the operations, we lose the capacity to compare corresponding operations between two algebras. For example, rings have two binary operations, and without an index to tell us which is addition and which is multiplication, which are often simplified to symbolic names, it would be impossible to determine whether two rings are comparable with regards to the superring/subring relation.

In addition, because we mention retracts later in the paper, we will define what a retract is here. Subalgebra $\S$ of algebra $\A$ is a retract of $\A$ when there exists a homomorphism $f:A \rightarrow S$ such that $f$ is the identity on $S$.

\subsection{Extension Monads}
If $\mathbb{A}$ is an algebra with nonstandard enlargement ${}^*\mathbb{A}$, then the extension monad of $\mathbb{A}$ is $\hat{\A} = \bigcap\{\mathbb{B} \leq {}^*\mathbb{A}$ such that $\mathbb{B}$ is internal and $\mathbb{A}$ is a subreduct of $\mathbb{B} \}$ \cite{Insall}. Note that $\hat{\A}$ is a subalgebra of ${}^*\mathbb{A}$, but it is generally external\textbf{---}`unrecognizable by the enlarged superstructure'\textbf{---}due to a limitation of language, similar to the monads of nonstandard analysis (encoding the notion of two standard points being `infinitesimally close') and to the following external objects used in nonstandard analysis:

\begin{enumerate}
    \item the finite part of a nonstandard enlargement of a normed vector space or of a metric space (see Keisler and others \cite{KeislerHyperrealLine});
    \item the near-standard part of a nonstandard enlargement of a normed vector space or of a metric space (see Duanmu and others \cite{Ergo});
    \item the nonstandard hulls of uniform spaces (see Alagu and Kala, among others \cite{AlaguKala}).
\end{enumerate}

\section{Results}
First, we consider concepts that will be more familiar to readers of classical algebra and compute some extension monads in that context. Then we briefly discuss the notion of a retract. We finish by finding more general methods of computation of extension monads in the universal algebraic setting.

\begin{theorem}\label{ExtensionMonadEq}
    For any algebra $\A$, the extension monad of $\A$ in ${}^*\A$ is given by:
    \begin{eqnarray}
    \hat{\A} &=& \bigcap \{\B \leq {}^*\A | \B \text{ is an internal, hyperfinitely generated, }\nonumber\\
    & &\text{algebra, such that }\A \text{ is a subreduct of } \B \}\nonumber\\
    &=& \left<\bigcup\{\left<F\right>^{(int)} \leq {}^*\A | F \subseteq A \text{ is finite}\}\right> \nonumber\\
    &=& \bigcup\{\left<F\right>^{(int)} \leq {}^*\A|F \subseteq \hat{\A} \text{ is hyperfinite}\}.\nonumber
    \end{eqnarray}
\end{theorem}
\begin{proof}
    These proofs have been adapted from \cite{Insall} to fit our broader assumptions of arbitrarily many operations.
    \newline
    \textbf{Proof of the first equality:} Since $\hat{\A} = \bigcap\{\mathbb{B} \leq {}^*\mathbb{A}$ such that $\mathbb{B}$ is an internal, hyperfinitely generated algebra and $\mathbb{A}$ is a subreduct of $\B \}$, we can clearly see that $\bigcap \{\B \leq {}^*\A | \B$ is internal and $\A$ is a subreduct of $ \B \} \geq \hat{\A}$.

    For the other direction, we know that every internal algebra that has $\A$ as a subreduct contains a hyperfinitely generated, internal algebra that has $\A$ as a subreduct.
    \newline
    \newline
    \textbf{Proof of the second equality:} Select $a \in {}^*\A$ such that $a \not\in \left<\bigcup\{\left<F\right>^{(int)} \leq {}^*\A| F \subseteq A\text{ is finite}\}\right>$. Define relation $r$ by $GrF \iff [a \not\in \left<F\right>^{(int)}$ and $G \subseteq F]$ for hyperfinite $G,F \subseteq {}^*\A$. Since $r$ is internal and concurrent on the set of all finite subsets of $\A$, by the concurrency principle there exists hyperfinite $F \subseteq {}^*\A$ such that $GrF$ for all finite $G \subseteq A$. Therefore, $A \subseteq F$ and $a \not\in \left<F\right>^{(int)}$. Consequently, $a \not\in \hat{\A}$ because $\hat{A} \leq \left<F\right>^{(int)}$. Thus, $\hat{\A} \leq \left<\bigcup\{\left<F\right>^{(int)} \leq {}^*\A | F \subseteq A\text{ is finite}\}\right>$.
    
    For the other direction, the fact that each $\left<F\right>^{(int)}$ for finite $F \subseteq A$ is contained in any internal algebra such that $\A$ is a subreduct of that algebra makes the proof trivial.
    \newline
    \newline
    \textbf{Proof of the third equality:} Clearly, $\hat{\A} \subseteq \bigcup\{\left<F\right>^{(int)} \leq {}^*\A|F \subseteq \hat{\A}$ is hyperfinite $\}$. Therefore, we simply need to show that the reverse inequality is true. Let $F = \{a_0, a_1, ..., a_{\nu-1}\}$ for $\nu \in {}^*\N$ and take finite subsets of $A$, $F_0, F_1, ..., F_{\nu-1}$, such that $a_0 \in \left<F_0\right>^{(int)}$ and $a_1 \in \left<F_1\right>^{(int)}$ and ... $a_{\nu-1} \in \left<F_{\nu-1}\right>^{(int)}$.
    \newline
    Because $E = F_0 \cup F_1 \cup ... \cup F_{\nu-1}$ is hyperfinite and a subset of $A$, we know that $E$ is finite. Hence, $F \subseteq \left<E\right>^{(int)} \subseteq \hat{\A}$ because $\hat{\A} = \left<\bigcup\{\left<F\right>^{(int)} \leq {}^*\A | F \subseteq A\text{ is finite} \}\right>$. This lets us conclude that $\bigcup\{\left<F\right>^{(int)} \leq {}^*\A|F \subseteq \hat{\A}$ is hyperfinite$\} \subseteq \hat{\A}$.
\end{proof}

\begin{corollary}\label{MoreExtensionMonadEq}
    \begin{eqnarray}
    \hat{\A} &=& \bigcup \{\left<F\right>^{(int)}\leq {}^*\A|F \subseteq \hat{\A} \hbox{ is finite}\}\nonumber\\
    &=& \left<\bigcup \{\left<F\right>^{(int)}\leq {}^*\A|F \subseteq \hat{\A} \hbox{ is finite}\}\right>\nonumber\\
    &=& \bigcup \{\left<F\right>^{(int)}\leq {}^*\A|F \subseteq \hat{\A} \hbox{ is hyperfinite in }{}^*\A \}.\nonumber
    \end{eqnarray}
\end{corollary}

\begin{proposition}
    If $V$ is a variety of algebras of the form $\A = (A, \mathcal{O})$ without infinitary operations in its language, then $\hat\cdot|_\mathcal{O}$ and ${}^*\cdot|_\mathcal{O}$ are operators on that variety.
\end{proposition}
\begin{proof}
    Let algebra $\A = (A, \mathcal{O})$ be in variety $V$ with a collection of characterizing equations, $\mathcal{C}$. Since each equation in $\mathcal{C}$ is of finite length and composed of symbols that are urelements, we see that $\mathcal{C} \subseteq {}^*\mathcal{C}$, and thus by transfer, ${}^*\A$ ${}^*$-satisfies all the equations in $\mathcal{C}$. This means ${}^*\A$ satisfies all of the equations in $\mathcal{C}$. Since varieties are closed under subalgebras and $\hat{\A} \leq {}^*\A$, we have that $\hat{\A}$ also satisfies the equations of $\mathcal{C}$. So we find that ${}^*\A|_\mathcal{O}$ and $\hat{\A}|_\mathcal{O}$ are algebras in the language of $V$ and satisfy all the equations in $\mathcal{C}$. Therefore, ${}^*\A|_\mathcal{O}, \hat{\A}|_\mathcal{O} \in V$.
\end{proof}

In \cite{Insall}, it is shown that $\hat{\A} = \bigcup \{\left<F\right>^{(int)} \leq {}^*\A | F \subseteq A$ is finite$\}$; however, it fails to always be true when infinitely many operations are allowed.

The following is both an extension and patching of the similar proof found in \cite{Insall}.

\begin{proposition}
    $\hat{\A} = \bigcup \{\left<F\right>^{(int)} \leq {}^*\A | F \subseteq A$ is finite$\}$.
\end{proposition}
\begin{proof}
    Let $\C = \bigcup \{\left<F\right>^{(int)} \leq {}^*\A | F \subseteq A$ is finite$\}$. Clearly, $\hat{\A} \geq \C$.
    Assume $\hat{\A} > \C$, then there exists some $x \in \hat{A} \setminus C$. Now define relation $r$ on $\mathcal{P}_{fin}(A)$, where $(H,K) \in r$ if and only if $x \notin \left < K \right >^{(int)}$ and $H \subseteq K$. This relation is concurrent because for any finite collection of finite subsets, their union will be a valid $r$-bound since $x$ is not in any internal subalgebra generated by a finite subset of $A$. Since $r$ is concurrent, there exists hyperfinite $L \subseteq {}^*A$ such that $x \notin \left < L \right >^{(int)}$ and $F \subseteq L$ for all finite subsets $F \subseteq A$. So $\left < L \right >^{(int)}$ is an internal subalgebra of ${}^*\A$ such that $\A$ is a subreduct of it. Therefore, $\hat{\A} \leq \left < L \right >^{(int)}$ and $x \notin \left < L \right >^{(int)}$, producing a contradiction. Thus, $\hat{\A} = \C$.
\end{proof}

Multiplication is the operation on the rationals that produces the much simpler extension monad, but a similar structure does appear in regard to addition. This paper is a portion of the doctoral dissertation \cite{Danielle}, with most of the results coming directly from there.

The next several results involve the computation of the extension monads of several related and well-known algebras, starting with the multiplicative group of nonzero rationals.

\begin{proposition}\label{MultGroupMonad}
    The extension monad of $\Q \setminus \{0\}$ as a group under multiplication is simply the group $\left <\Q  \setminus \{0\} \right>^{(int)}_{({}^*\Q;\cdot^{{}^*\Q})}$.
\end{proposition}
\begin{proof}
    Let $\M = \left<\Q  \setminus \{0\} \right>^{(int)}_{({}^*\Q;\cdot^{{}^*\Q})}$ and let $\A$ be an internal extension of $\Q  \setminus \{0\}$ contained in ${}^*(\Q  \setminus \{0\})$. Clearly, $\M \leq \A$ by the closure of algebras, so $M \subseteq \widehat{Q  \setminus \{0\}}$. Now let $\zeta \in {}^*(\Q  \setminus \{0\})$ be such that $\zeta \notin \M$ and $\zeta = \frac{\eta}{\nu}$ is the most simplified form. Because $\M$ is the closure of the rational numbers by hyperfinite multiplication, either $\eta$ or $\nu$ must have a prime factor $\psi \in {}^*\Z \setminus \Z$. Also since the internal group $\left<\{\frac{x}{y}| x, y \in \{-\psi +1, -\psi + 2, ..., -1, 1, ..., \psi -2, \psi -1\}\}\right>^{(int)}_{({}^*\Q;\cdot^{{}^*\Q})}$ excludes $\zeta$ and contains $\Q  \setminus \{0\}$, we have $\M = \widehat{\Q \setminus \{0\}}$.
\end{proof}

As mentioned above, addition is the more complex operation.

\begin{proposition}\label{AddGroupMonad}
    The extension monad of $\Q$ as a group under addition is simply the group whose underlying set is $\left \{\frac{x}{y}|x \in {}^*\Z, y \in \left <\Z\setminus\{0\}\right >^{(int)}_{({}^*\Z;\cdot^{{}^*\Z})} \right\}$.
\end{proposition}
\begin{proof}
    Let $x_1,x_2 \in \Z$ and $y_1,y_2 \in \Z \setminus \{0\}$. We see that $\frac{x_1}{y_1}-\frac{x_2}{y_2} = \frac{x_1y_2-x_2y_1}{y_1y_2}$. We also know that $x_1y_2-x_2y_2 \in {}^*\Z$ due to closure under addition and additive inverses and that $y_1y_2 \in \left<\Z\setminus\{0\}\right >^{(int)}_{({}^*\Z;\cdot^{{}^*\Z})}$ by the definition of the closure. It is clear that $\frac{0}{1}=0$ is in this set, as well, giving us closure under each of our group operations.
    
    We can also show that $\hat{\Q}$ excludes any denominator outside of \newline $\left<\Z\setminus\{0\}\right >^{(int)}_{({}^*\Z;\cdot^{{}^*\Z})}$ if we have $\eta \in {}^*\Z\setminus \left<\Z\setminus\{0\}\right >^{(int)}_{({}^*\Z;\cdot^{{}^*\Z})}$, we let $\psi$ be some infinite prime factor of $\eta$, and we consequently take an internal group whose underlying set is of the form $\{\frac{x}{y}|x \in {}^*\Z, y \in \left<{}^*\Z \setminus \{\psi \} \cap (-\eta +1, \eta -1)\right>^{(int)}_{({}^*\Z;\cdot^{{}^*\Z})} \}$. This is also closed under the group operations and excludes any fraction for which $\eta$ is in the denominator.
\end{proof}

It turns out that the underlying set of the extension monad of the rational numbers as a ring is the same as the extension monad of the additive group reduct.

\begin{proposition}\label{QRingMonad}
    The extension monad of the rational numbers as a ring has the same underlying set as the extension monad of the rational numbers as an additive group.
\end{proposition}
\begin{proof}
    Let $\Q$ be the rational numbers as a ring with the ordinary operations, and let $R = \left \{\frac{x}{y}|x \in {}^*\Z, y \in \left <\Z\setminus\{0\}\right >^{(int)}_{({}^*\Z;\cdot^{{}^*\Z})} \right\}$. Clearly, set $R$ is closed under addition, and it is also an easy exercise to see that it is closed under multiplication.

    It should be clear that $\R$ is a subring of any internal subring of ${}^*\Q$ that contains $Q$ by the proof of \ref{AddGroupMonad}. Now we merely show that for any element $\eta \notin {}^*Q \setminus R$, that element is not in $\hat{\Q}$. We find the same excluded elements as the previous proof, and the set $\{\frac{x}{y}|x \in {}^*\Z, y \in \left<{}^*\Z \setminus \{\psi \} \cap (-\eta +1, \eta -1)\right>^{(int)}_{({}^*\Z;\cdot^{{}^*\Z})} \}$ is closed under the ring operations, giving us $R = \hat{Q}$.
\end{proof}

The following corollary shows a noncommutative example. Even though finite matrices are not urelements, they are easily representable as urelements, allowing us to make the following conclusion.

\begin{corollary}\label{}
    The extension monad of the ring of n $\times$ n matrices whose entries are members of $\Q$ is the ring of n $\times$ n matrices whose entries are members of $\hat{\Q}$.
\end{corollary}
\begin{proof}
    Since, additively, the ring of n $\times$ n matrices whose entries are members of $\Q$ is isomorphic to the $n^2$ direct power of $\Q$, the extension monad of this additive group is $\hat{\Q}^{n\times n}$. Multiplication also has the same extension monad because any element not in $\hat{\Q}^{n\times n}$ will be excluded from some $\{\frac{x}{y}|x \in {}^*\Z, y \in \left<{}^*\Z \setminus \{\psi \} \cap (-\eta +1, \eta -1)\right>^{(int)}_{({}^*\Z;\cdot^{{}^*\Z})} \}^{n\text{x}n}$ for the same reasons as in \ref{QRingMonad}.
\end{proof}

The following corollary shows a semigroup example.

\begin{corollary}\label{}
    The underlying set of the extension monad of the additive semigroup of nonnegative rational numbers is $\hat{Q} \cap {}^*(0,\infty)$.
\end{corollary}
\begin{proof}
    Set $\hat{Q} \cap {}^*(0,\infty)$ is closed under addition for the same reasons that the rational group is, just with no mention of subtraction or inverses.

    We can also show that $\hat{\Q} \cap {}^*(0,\infty)$ excludes any denominator outside of \newline $\left<\Z\setminus\{0\}\right >^{(int)}_{({}^*\Z;\cdot^{{}^*\Z})} \cap {}^*(0,\infty)$ if we have $\eta \in {}^*\Z\setminus \left<\Z\setminus\{0\}\right >^{(int)}_{({}^*\Z;\cdot^{{}^*\Z})}$, and we let $\psi$ be some infinite prime factor of $\eta$, and we consequently take an internal group whose underlying set is of the form $\{\frac{x}{y}|x \in {}^*\Z, y \in \left<{}^*\Z \setminus \{\psi \} \cap (-\eta +1, \eta -1)\right>^{(int)}_{({}^*\Z;\cdot^{{}^*\Z})} \} \cap {}^*(0,\infty)$. This is also closed under the group operations and excludes any fraction for which $\eta$ is in the denominator.
\end{proof}

The following corollary shows a nonassociative example.

\begin{corollary}\label{}
    The extension monad of the algebra of three-space vectors of elements of $Q$ with the operation defined as the cross product $o((a,b,c),(x,y,z)) = (bz-cy,-az+cx,ay-bx)$ is the algebra whose underlying set is $(\hat{Q})^3$.
\end{corollary}
\begin{proof}
    By the same arguments previously used and the fact that $\hat{Q}$ is closed under the usual ring operations, we see that $(\hat{Q})^3$ is closed under the cross product.

    If there is an element $x \notin (\hat{Q})^3$, then by a similar argument as \ref{QRingMonad}, we know there exists a nonstandard prime in the denominator of one of the entries, and so we can construct an algebra whose underlying set is $(\{\frac{x}{y}|x \in {}^*\Z\cap {}^*(0,\infty), y \in \left<{}^*\Z \cap {}^*(0,\infty) \setminus \{\psi \} \cap (0, \eta -1)\right>^{(int)}_{({}^*\Z;\cdot^{{}^*\Z})} \})^3$ such that $x$ is not in that algebra. This shows us our desired result.
\end{proof}

Another intuitive train of thought about extension monads is that for a subalgebra, its extension monad will simply be the intersection of its enlargement and the extension monad of the superalgebra.

\begin{proposition}\label{}
    If $\A \leq \B$, then ${}\hat{\A} = \hat{\B}\cap{}^*\A$.
\end{proposition}
\begin{proof}
    Clearly, ${}\hat{\A} \leq \hat{\B}\cap{}^*\A$ because ${}\hat{\A} \leq \hat{\B}$ and ${}\hat{\A} \leq {}^*\A$. Take $x \in (\hat{\B} \cap {}^*\A)$. Then, $\forall R \in \{S| \B \leq \S \ {}^*\hspace{-.06in}\leq {}^*\B\}$, and we have $x \in R \cap {}^*A$. Since $B \cap {}^*A = A$, we can say that $\A \leq \R\cap{}^*\A$, and for all $\T$ such that $\A \leq \T \ {}^*\hspace{-.06in}\leq {}^*\A$, we can find an algebra $\R' \in \{S| \B \leq \S \ {}^*\hspace{-.06in}\leq {}^*\B\}$ such that $\R'\cap {}^*\A = \T$ by taking $\left < T \cup (B\setminus {}^*A) \right >^{(int)}$ because $\T$ is an internal algebra and, hence, is closed as a subalgebra. Therefore, we only generate elements of ${}^*\B \setminus {}^*\A$, and we have that $x \in \hat{\A}$.
\end{proof}


\begin{lemma}\label{FiniteCheckEq}
    For algebra $\A$ generated by finite $F$, we have $\hat{\A} = \check{F}$.
\end{lemma}
\begin{proof}
    Since $\A$ is finitely generated, we know that $\hat{\A} = {}^*\A$. Therefore, any internal algebra containing $\A$ equals ${}^*\A$. Since $F$ is internal because it is finite, $\check{F}$ will also be internal, and because $F$ generates $\A$, we have that $\A \leq \check{F}$, which is the desired result.
\end{proof}

\begin{theorem}\label{}
    For algebra $\A$ generated by $G$, we have $\hat{\A} = \check{G}$.
\end{theorem}
\begin{proof}
    Assume algebra $\A$ is generated by set $G$. Then, ${}^*\A$ is ${}^*$-generated by ${}^*G$, which means that the operations are iterated hyperfinitely many times as opposed to finitely many times. Take some internal subalgebra of ${}^*\A$ that contains $A$. Clearly, it will also contain $G$, and because it is an internal algebra, it will be closed under hyperfinitely many iterations of its operations. Thus, $\hat{\A} \supseteq \check{G} $.
    
    Now let $x \in \hat{\A} \setminus \check{G}$. Then for every hyperfinite set $H$ containing $G$, we have $x \in \left < H \right >^{(int)}$. Applying overspill to relation $r$ of the hyperfinite subsets of ${}^*\A$, where $H_1 r H_2$ if and only if $x \in \left < H_2 \right >^{(int)}$ and $H_2 \subseteq H_1$, which is concurrent on the subsets containing $G$, we see that there exists an internal hyperfinite subset of ${}^*\A$, $H'$, which is contained in every hyperfinite subset of ${}^*\A$ that contains $G$. Since $G$ is the intersection of all these sets, we have that $H' \subseteq G$, and since $G$ is external, we know $H'$ is finite and $x \in \left <H'\right >^{(int)}$. So by $\ref{FiniteCheckEq}$ we have that $x \in \check{H'} \leq \check{G}$.
\end{proof}

Now that we have covered some concrete examples, we move to broader characterizations, starting with results more closely mirroring the work of \cite{Insall} and then moving on to more original results.

The below results are also adapted from \cite{Insall}, allowing for infinitely many operations.

\begin{theorem}\label{LocalFiniteEquiv}
        An algebra is locally finite if and only if $\hat{\A}|_\mathcal{O} = \A$.
\end{theorem}
\begin{proof}
    This proof is adapted from \cite{Insall}. However, we will subreducts to account for the possibility of infinitely many operations.

    $\mathbf{\Longrightarrow}$: If $\A$ is locally finite and $F \subseteq A$ is finite, then $\left<F\right>^{(int)}$ is an expansion of $\left<F\right>$. Therefore, $\left<F\right> \leq \left<F\right>^{(int)}|_\mathcal{O} = \bigcap\{\E \leq {}^*\A| \E$ is internal, and $F \subseteq E\}|_\mathcal{O} \leq \left<F\right>$ because 
    \[
    \bigcap\{\E \leq {}^*\A| \E\hbox{ is internal and }F \subseteq E\} \subseteq \left<F\right>
    \]
    for finite $F$. It follows that $\left<F\right> = \left<F\right>^{(int)}|_\mathcal{O}$, and so $\hat{\A} = \left<\bigcup\{\left<F\right>^{(int)} \leq {}^*\A|F \subseteq A\text{ is finite}\}\right>^{(int)}$ $= \left<\bigcup \{\left<F\right> \leq {}^*\A|_\mathcal{O}|F \subseteq A\text{ is finite}\}\right>^{(int)}$ since the underlying sets are the same. Thus,
    \begin{eqnarray}
    \hat{\A}|_\mathcal{O} &=& \left<\bigcup\{\left<F\right>^{(int)} \leq {}^*\A|F \subseteq A \text{ is finite} \}\right>^{(int)}|_\mathcal{O} \nonumber\\
    &=& \left<\bigcup\{\left<F\right>^{(int)}|_\mathcal{O} \leq {}^*\A|_\mathcal{O}|F \subseteq A \text{ is finite}\}\right>^{(int)}|_\mathcal{O} \nonumber\\
     &=& \left<\bigcup\{\left<F\right> \leq {}^*\A|_\mathcal{O}|F \subseteq A \text{ is finite}\}\right>^{(int)}|_\mathcal{O} \nonumber\\
    &\geq& \left<\bigcup\{\left<F\right> \leq \A|F \subseteq A \text{ is finite}\} \right>\nonumber\\
    &=& \A.\nonumber
    \end{eqnarray}
    
    Now we can prove that $\hat{\A}|_\mathcal{O} \leq \A$. Assume there exists $\nu$ such that $\nu \in \hat{A}$ and $\nu \notin A$. Let $\B$ be an internal algebra such that $\A$ is a subreduct of $\B$. Then $\nu \in B$. Note that there exists operation $o_\omega\in{}^*\mathcal{O}$ and a set of operands $F_\omega = \{a_0,a_1,...,a_{\omega-1}\} \subseteq \bigcup \left<F\right>^{(int)} = A$ because $\A$ is locally finite such that $o_\omega(a_0,a_1,...,a_{\omega-1}) = \nu \notin A$. If no such $\nu$ exists, then every operation with every possible set of operands must return an element in $A$, and thus, by Theorem 3, we can conclude that $\hat{\A}|_\mathcal{O} = \A$.
    
    If we take $F_0,F_1,...,F_{\omega-1} \subseteq A$ to be finite sets such that $a_0 \in F_0, ..., a_{\omega-1} \in F_{\omega-1}$, then since $F_\omega$ is hyperfinite, we also know $E = F_0 \cup F_1 \cup ... \cup F_{\omega-1} \subseteq A$ is hyperfinite. This tells us that $E$ is finite, and thus, $\nu \in \left<E\right>^{(int)}|_\mathcal{O} = \left<E\right> \leq \A$. Moreover, we have that $\hat{\A}|_\mathcal{O} \leq \A$. Hence, $\hat{\A}|_\mathcal{O} = \A$.

    $\mathbf{\Longleftarrow}$: If $\hat{\A}|_\mathcal{O} = \A$ and $F \subseteq A$ is finite, then $\left<F\right>^{(int)} \leq \hat{\A}$ by the definition of $\hat{\A}$. Now, ${}^*\left<F\right>|_\mathcal{O} \leq  \left<F\right>^{(int)}|_\mathcal{O} \leq \hat{\A}|_\mathcal{O} = \A$ because any element of $\left<F\right>$ is also in $\left<F\right>^{(int)}$. This implies ${}^*\left<F\right>$ has only standard elements, which implies that, due to overspill, $\left<F\right>$ has only a finite number of elements. Therefore, $\A$ is locally finite.
\end{proof}


\begin{lemma}\label{}
    Let $\A$ and $\B$ be algebras and $\P$ be their direct product, then $range({}^*(\rho_\A|_{\hat\P})) = \hat\A$.
\end{lemma}
\begin{proof}
    Assume $x \in range({}^*(\rho_\A|_{\hat\P})) \setminus \hat\A$, so there must exist some internal $\I$ such that $\A \leq \I$ and $x \notin \I$. There must also exist $y \in \hat\P$ such that ${}^*(\rho_\A)(y) = x$. By the properties of homomorphisms, if $\S$ is any internal subalgebra of ${}^*\P$ containing $\P$, we have that ${}^*(\rho_\A)[\S]$ is an internal subalgebra of ${}^*\A$ that contains $\A$, Thus, by the closure of algebras under intersection, we see that $range({}^*(\rho_\A|_{\hat\P})) \supseteq \hat\A$. Taking $({}^*\rho_\A|_{\hat\P})^{-1}[\I]$, we get an internal algebra of ${}^*\P$ that contains $\P$ and thus $y \in ({}^*\rho_\A|_{\hat\P})^{-1}[\I]$ and so $x \in \I$. Thus giving us our contradiction and so $range({}^*(\rho_\A|_{\hat\P})) \subseteq \hat\A$.
\end{proof}

For the proof of the next theorem, it is important to note that when we are discussing ${}^*$products, we are not speaking of the internal direct product most commonly used in group theory.

\begin{theorem}\label{UrProdMonad}
    For a finite collection of algebras, $\{\A_1, \A_2, ..., \A_n\}$, in the same variety $V$, the extension monad of the direct product of those algebras will be the direct product of the extension monads of those algebras.
\end{theorem}
\begin{proof}
    We first consider the case with two algebras because induction easily shows us the rest. Let $V$ be a variety, $\A, \B \in V$, and $\rho_\A, \rho_\B$ be the projection homomorphisms from the product $\P = \Pi\{\A,\B\}$ to $\A$ and $\B$, respectively. By transfer, we know that ${}^*\P$ has projections ${}^*(\rho_\A) = \rho_{{}^*\A}$ and ${}^*(\rho_\B) = \rho_{{}^*\B}$ which map onto ${}^*\A$ and ${}^*\B$, respectively. 
    
    Now let the product of the extension monads be denoted by $\E$ with projections $\rho_{\hat\A}$ and $\rho_{\hat\B}$ mapping onto $\hat\A$ and $\hat\B$, also respectively. We will show that $\hat\P$ is also a direct product of $\hat\A$ and $\hat\B$. We know $\E$ is the direct product of $\bigcap\{\S|\A \leq \S \leq {}^*\A, \S \text{ internal}\}$ and $\bigcap\{\T|\B \leq \T \leq {}^*\B, \T \text{ internal}\}$. We also know $\hat\P = \bigcap \{\U|\P \leq \U \leq {}^*\P,\U \text{ internal}\}$ and $\bigcap \{ \V| \V \text{ is a } {}^*\text{product of some internal } \Y,\Z \text{ for } \A \leq \Z \leq {}^*\A, \B \leq \Y \leq {}^*\B\}$, which contains an isomorphic copy of $\P$ and, thus, also contains an isomorphic copy of $\hat\P$. 
    
    Given any internal $\U$ containing $\hat\P$, let $\D$ be an internal finitely generated subalgebra of $\Pi\{{}^*\A,{}^*\B\}$ whose finite generating set $F$ is contained in $\Pi\{\A,\B\}$. Let $\K \leq {}^*\A$ and $\M \leq {}^*\B$ be $\rho_{{}^*\A}[\D]$ and $\rho_{{}^*\B}[\D]$, respectively, and be generated internally by finite subsets $\rho_\A[F]$ and $\rho_\B[F]$, respectively. We have that $\D$ is internally embeddable into the internal algebra generated by $\Pi\{\rho_\A[F],\rho_\B[F]\}$, which we will call $\C$. Since we have $F \subseteq \Pi\{\rho_\A[F],\rho_\B[F]\}$, we now look at $\C$, which is embeddable into $\U$. Since $\U$ contains isomorphic copies of every possible internal and finitely generated $\D$ in $\Pi\{{}^*\A,{}^*\B\}$, we see that\textbf{---}by concurrency on relation $r$ between finitely generated subalgebras of $\Pi\{{}^*\A,{}^*\B\}$ whose generating sets are subsets of $\Pi\{\A,\B\}$, where $\F r \G$ if and only if $F \subseteq G$ and $\G$ is internally embeddable into $\U$\textbf{---}there exists an internal hyperfinitely generated subalgebra of $\Pi\{{}^*\A,{}^*\B\}$ that is embeddable into $\U$ and contains $\bigcup\{\D|\D = \langle F\rangle^{(int)}, F \in \mathcal{F}(P), \D \leq \Pi\{{}^*\A,{}^*\B\}\} = \E$. Therefore, $\E$ is internally embeddable into $\U$. Since this holds for every $\U$, we see that $\E$ is embeddable into $\bigcap\{\U|\U \leq {}^*\P, \U \text{ internal}, \hat\P \leq \U\} = \hat\P$.
    
    Now that we have shown that $\E$ is embeddable into $\hat\P$, we have the following commutative diagram:

\[\begin{tikzcd}
	&& \E \\
	\\
	{\hat\A} && {\hat\P} && {\hat\B}
	\arrow["{\rho_{\hat\A}}", two heads, from=1-3, to=3-1]
	\arrow["\theta"{description}, dashed, tail, from=1-3, to=3-3]
	\arrow["{\rho_{\hat\B}}"', two heads, from=1-3, to=3-5]
	\arrow["{\rho_{{}^*\A}|_{\hat\P}}"', two heads, from=3-3, to=3-1]
	\arrow["{\rho_{{}^*\B}|_{\hat\P}}", two heads, from=3-3, to=3-5]
\end{tikzcd}\]
    where $\theta$ is an embedding.

    Therefore, we simply show that we can always complete the following diagram so that it commutes for an arbitrary algebra $\G$ in the same variety:

\[\begin{tikzcd}
	&& \G \\
	\\
	{\hat\A} && {\hat\P} && {\hat\B}
	\arrow["{\psi_{\hat\A}}", from=1-3, to=3-1]
	\arrow["\gamma"{description}, dashed, from=1-3, to=3-3]
	\arrow["{\psi_{\hat\B}}"', from=1-3, to=3-5]
	\arrow["{\rho_{{}^*\A}|_{\hat\P}}"', two heads, from=3-3, to=3-1]
	\arrow["{\rho_{{}^*\B}|_{\hat\P}}", two heads, from=3-3, to=3-5]
\end{tikzcd}\]
    which we see from the following diagram because $\E = \Pi\{\hat\A,\hat\B\}$:

\[\begin{tikzcd}
	&& \G \\
	\\
	{\hat\A} && \E && {\hat\B} \\
	\\
	&& {\hat\P}
	\arrow["{\psi_{\hat\A}}", from=1-3, to=3-1]
	\arrow["\epsilon"{description}, dashed, from=1-3, to=3-3]
	\arrow["{\psi_{\hat\B}}"', from=1-3, to=3-5]
	\arrow["{\rho_{\hat\A}}"', two heads, from=3-3, to=3-1]
	\arrow["{\rho_{\hat\B}}", two heads, from=3-3, to=3-5]
	\arrow["\theta"{description}, tail, from=3-3, to=5-3]
	\arrow["{\rho_{{}^*\A}|_{\hat\P}}"', two heads, from=5-3, to=3-1]
	\arrow["{\rho_{{}^*\B}|_{\hat\P}}", two heads, from=5-3, to=3-5]
\end{tikzcd}\]
    Since we can always find $\gamma = \theta\circ\epsilon$, we see that $\hat\P$ is also a direct product of $\hat\A$ and $\hat\B$.
\end{proof}

\begin{theorem}\label{FinitelyGeneratedEquiv}
    Let $\A$ be an algebra, then the following are equivalent:
    \begin{itemize}
        \item[a)] $\A$ is finitely generated
        \item[b)] $\hat{\A} = {}^*\A$
        \item[c)] $\hat{\A}$ is internal.
    \end{itemize}
\end{theorem}
\begin{proof}
    This proof is adapted from the argument provided by \cite{Insall}. For finite $F \subseteq A$, we know that ${}^*\left<F\right>=\left<F\right>^{(int)}$ because, given $\left<F\right> = (G, \mathcal{O})$, we have $x \in G \iff [x \in F$ or $\exists n\in \N\hbox{ such that for some } n\hbox{-ary } o \in \mathcal{O}$ and some $s \in (G\setminus\{x\})^n$ we have $o(s) = x]$. 

    By transfer and since ${}^*F = F$, we can state that $x \in {}^*G \iff [x \in F$ or $\exists n\in {}^*\N $ such that for some $ n\hbox{-ary } o \in {}^*\mathcal{O}$ and some $ s \in {}^*(G\setminus\{x\})^n = ({}^*G\setminus\{x\})^n$ we have $o(s) = x]$. 

    The previous transfer tells us that the underlying set of $\left<F\right>^{(int)}$ is just the ${}^*$-transform of the underlying set of $\left<F\right>$. And therefore, ${}^*\langle F\rangle={}^*(G,\mathcal{O})=({}^*G,{}^*\mathcal{O})=\left<F\right>^{(int)}$.
    
    $\mathbf{a \implies [b \& c]}$: If $\A = \left<F\right>$ for some finite $F \subseteq A$, then $\left<F\right>^{(int)} = {}^*\A$ by transfer. Thus, ${}^*\A = \bigcap \{\B \leq {}^*\A|\B $ is hyperfinitely generated and contains $ F\}$ $\subseteq \bigcap \{\B \leq {}^*\A|\B$ is an internal hyperfinitely generated algebra such that $\A$ is a subreduct of $\B\} = \hat{\A}$.

    $\mathbf{b \implies a}$: If $\hat{\A} = {}^*\A$, then for some hyperfinite $F \subseteq {}^*\A$, we have $\left<F\right>^{(int)} \supseteq \hat{\A} = {}^*\A$. This implies that ${}^*\A$ is hyperfinitely generated. Hence, by transfer, we know that $\A$ is finitely generated.
    
    $\mathbf{c \implies a}$: Assume ${}^*\A \neq \hat{\A}$ and let $\B < {}^*\A$ be internal with $A \subset B$. Since ${}^*\A \neq \hat{\A}$, we know that $\A$ is not finitely generated. 

    We want to show that $\hat{\A} \neq \B$, so let $r$ be the relation on the hyperfinite subsets of $\B$ such that $FrG \iff [F \subseteq G\ \& \left<G\right>^{(int)}\ <\ \B$]. Here, $r$ is clearly defined on the finite subsets of $\A$ and is, in fact, internal and concurrent, as well. Thus, by the concurrency principle, we know there exists hyperfinite $G \subseteq \B$ such that $FrG$ for all finite $F \subseteq A$. This tells us that $A \subseteq G$, and we have $\left<G\right>^{(int)} \ < \ \B$. Therefore, $\hat{\A} \neq \B$. So $\hat{\A}$ is not internal, and since we know $\neg a \implies \neg b$ and we just proved $\neg b \implies \neg c$, where $\neg$ means the logical negation of a statement, we can conclude that $c \implies a$.
\end{proof}


Extension monads behave well with respect to direct products but not as well with respect to coproducts or direct sums.

\begin{example}\label{}
    For two two-element groups, $\A$ and $\B$, in the variety of groups, we have $\hat{\A} \oplus \hat{\B} < \widehat{\A \oplus \B}$. This is clear because, given $\A$ and $\B$ are finite, we know $\hat{\A} = \A$ and $\hat{\B} = \B$. Therefore, $\hat{\A} \oplus \hat{\B} = \A \oplus \B$, which is infinite because nonidentity elements act freely with respect to one another. Since $\A \oplus \B$ is finitely generated, $\widehat{\A \oplus \B} = {}^*(\A \oplus \B)$. Moreover, because $\A \oplus \B$ is infinite, we know $\widehat{\A \oplus \B} = {}^*(\A \oplus \B) > \A \oplus \B = \hat{\A} \oplus \hat{\B}$. However, this lack of equality is not always true, even in the variety of groups, because the direct sum of two trivial groups is itself trivial, and thus, so is the extension monad.
\end{example}
\newline

A natural question that may arise from this example is how an extension monad of an extension of an algebra is related to the previous settings. The following few results display the behaviors of these types of field extensions of groups.

\begin{theorem}\label{}
    The extension monad of the additive group of rationals extended as a field extension by $\pi$, $\Q(\pi)$, is the group \newline $S = \left \{ \Sigma_{j=0}^n q_j\pi^j | n \in {}^*\N, q_j \in \hat{\Q} \right \}$ under addition.
\end{theorem}
\begin{proof}
    Consider $\Sigma_{j=0}^n q_j\pi^j - \Sigma_{i=0}^m r_i\pi^i$. If we add extra zeros when necessary, we get $\Sigma_{l=0}^{max(n,m)} (q_l - r_l)\pi^l$, which is clearly an element of $S$.
    
    Let $\alpha \in {}^*(\Q(\pi)),$ which is not in $S$, then there are two cases. In the first case, we assume that in every expression of the form $\Sigma_{j \in H} q_j\pi^j = \alpha$ for $H \in {}^*\mathcal{F}({}^*\Q)$, for some infinite prime hyperfinite $\psi$ and some $\eta \in {}^*\N$, we get $\Sigma_{j \in H} q_j\pi^j = \Sigma_{i \in H \setminus \{\eta\}} q_i\pi^i + \frac{\xi}{\phi \psi}\pi^\eta$. We can construct the group 
    
    \[
    \begin{array}{l}
    \left \{\phantom{{}^{(int)}_{({}^*\Z;\cdot^{{}^*\Z})}}\hspace{-.35in} \Sigma_{j \in H} q_j\pi^j | \forall H \in {}^*\mathcal{F}({}^*\N), \left \{q_j|j \in H \right \} \in {}^*\mathcal{P} ({}^*\Q)\cap \right. \\ 
    \left. \left\{\frac{x}{y}|x \in {}^*\Z, y \in \left<{}^*\Z \setminus \left \{\psi \right \} \cap (-\phi\psi +1, \phi\psi -1)\right>^{(int)}_{({}^*\Z;\cdot^{{}^*\Z})} \right\} \right\},
    \end{array}
    \]
\newline
which excludes $\alpha$ and is an internal group since the group \newline $\left \{\frac{x}{y}|x \in {}^*\Z, y \in \left <\Z\setminus\{0\}\right >^{(int)}_{({}^*\Z;\cdot^{{}^*\Z})} \right\}$ is closed under hyperfinite addition.
    
    In the second case, we assume there exists a summand in every expression of the form $r\pi^\gamma$ for $r \neq 0$ and fixed $\gamma \in {}^*\N\setminus\N$. We can also construct the group $\{\Sigma_{i < \gamma} q_i \pi^i | q_i \in {}^*\Q \}$, which is closed under hyperfinite addition, because we can transfer the statement $\forall n \in \N, \forall \{q_i| 0 \leq i \leq n\} \Sigma_0^n q_i\pi^i \neq \pi^{n+1}$ since $\pi$ is transcendental. This set clearly satisfies all the other group closure properties and is internal because the summations do not violate the overspill principle.
    
    Let $s = \Sigma_{j=0}^n q_j\pi^j \in S$ for $n \in {}*\N$ and let $\A \leq {}^*(\Q(\pi))$ be internal such that $Q(\pi) \subseteq A$. Since $\A$ is internal, it is closed under hyperfinite sums and, thus, finite sums. Since $q_j \in \hat{\Q}$, the product $q_j \pi^j$ is in $\A$, and because $\pi^m \in \Q(\pi) \forall m \in \N$, we see that $S \subseteq A$.
\end{proof}

For extensions by one element, as seen above, the behavior is quite nice. Another question then arises: what happens with arbitrarily many extending elements? As shown below, the behavior is not too dissimilar from single-element extensions. For clarity, a \textbf{base} of an extension of a field is a minimal subset of the extension such that extending the field by that set results in the extension.

\begin{theorem}\label{}
    Let $B$ be a base for the extension of $\Q$ to $\R$. The extension monad of the real numbers as a group under addition is the group whose underlying set is $\{ \Sigma_{i \in I} \Sigma_{j=0}^{n_i} q_{i,j}\theta_i^j | (\forall i \in I)(  \theta_i \in B, q_{i,j} \in \hat{\Q}), I \in {}^*\mathcal{F}({}^*\N), \{n_i | i\in I\} \subseteq \N\}$.
\end{theorem}
\begin{proof}
    For each individual element of $B$, an adjustment of the previous proof suffices for closure and exclusion of other elements. Therefore, since $\Q(B) = \R$ and since every internal algebra $\A \subseteq {}^*\R$ that contains $\R$ is closed under hyperfinite sums, we find that every sum of the form $\Sigma_{i \in I} \Sigma_{j=0}^{n_i} q_{i,j}\theta_i^j$ is in $\A$.
\end{proof}

This result can be generalized to any group that can be expanded to a field.

\begin{theorem}\label{}
    The extension monad of infinite group $\G$ that is an additive reduct of field $\F$ is the extension monad of prime subfield $\P$ extended by a base of $\F$ over $\P$. Which is to say, for some base $B$ of the extension of $\F$ over $\G$ as a field extension:
    
    $\S = \left \{ \Sigma_{i \in \gamma}\Sigma_{j=0}^{n_i} p_{i,j}a_i^j | \gamma \in {}^*Fin(\N), n \in \N, a_i \in B, p_{i,j} \in \hat{\P} \right \} =\hat{\F}$.
\end{theorem}
\begin{proof}
    Closure under additive group operations behaves the same as in the previous proofs, so $\S$ is a group. The arguments for equality of the underlying sets will also go similarly, but some of the details will be covered here. Say we have some element of the form $s = \Sigma_{i \in \gamma}\Sigma_{j=0}^{n_i} p_{i,j}a_i^j = \Sigma_{k \in \gamma}\Sigma_{j=0}^{n_i} r_{i,j}a_i^j + p^\prime a_{i^\prime}^{j^\prime}$ for which $p^\prime \notin \hat{Q}$ for $i^\prime \in \gamma$, $j^\prime \leq n_i$, and $r_{i^\prime, j^\prime} = 0$. Since $p^\prime \notin \hat{P}$, there exists internal additive group $\A \leq {}^*\P$ such that $p^\prime \notin A$ and $P \subseteq A$. Hence, the internal group $\left \{ \Sigma_{i \in \gamma}\Sigma_{j \in \eta_i} p_{i,j}a_i^j | \gamma, \eta_i \in {}^*Fin(\N), a_i \in {}^*F \setminus {}^*P, p_{i,j} \in A \right \}$ excludes $s$ and contains $\F$.
    
    Say we have an element of the form $t = \Sigma_{i \in \gamma}\Sigma_{j \in \eta_i} p_{i,j}a_i^j = \Sigma_{k \in \gamma}\Sigma_{l \in \eta_i} r_{i,j}a_i^j = p^\prime a_{i^\prime}^{j^\prime}$ for $p^\prime \neq 0$, $r_{i^\prime, j^\prime} = 0$, $i^\prime \in \gamma$, and $j^\prime \notin {}^*\N$. The internal group 
    
    \noindent $\left \{ \Sigma_{i \in \gamma}\Sigma_{j \in \nu_i} p_{i,j}a_i^j | \gamma \in {}^*Fin(\N), \nu_i \leq j^\prime, a_i \in {}^*F \setminus {}^*P, p_{i,j} \in {}^*F \right \}$ 
    excludes $t$ and contains $\F$, so $\S \geq \hat{\P}$.
    
    To show $\S \leq \hat{\P}$, take any internal $\A \leq {}^*\F$ that contains $\F$. Since $\P \leq \F$, we have $\hat{P} \subseteq \hat{F} \subseteq A$, so the product $q_{i,j} a_i^j$ for $q_{i,j} \in \hat{P}, a_i \in B$, and $j \in \N$ is in $\A$. Therefore, because internal groups are closed under hyperfinite sums, we get our result.
\end{proof}

We now transition our attention to determining when algebras are the retracts of their enlargement or their extension monad. Many of the following results and computations are relatively simple but set the stage for more general statements.

\begin{example}\label{}
    $\hat{\Q}$ is not a retract of ${}^*\Q$ as a ring. Take any infinite positive prime element $\rho$. It is clear that $\frac{1}{\rho} \notin \hat{Q}$, but we have $\rho * \frac{1}{\rho} = 1$, so if we were to assume that $f: {}^*Q \rightarrow \hat{Q}$ is a retract, we see that $1 = f(1) = f(\rho * \frac{1}{\rho}) = f(\rho) f(\frac{1}{\rho}) = \rho * f(\frac{1}{\rho})$. This shows us that $f(\frac{1}{\rho}) = \frac{1}{\rho}$, which produces a contradiction.
\end{example}
\newline

The next several results examine the relation between algebras that are retracts of either their enlargement or extension monad and the resulting relation between their products, sums, or intersections and the enlargement or extension monad of those products, sums, or intersections. We will significantly improve these results later, easing into the better result by the end of the section.

\begin{proposition}\label{}
If $\mathbb{A}$ is a retract of ${}^*\mathbb{A}$ and $\mathbb{B}$ is a retract of ${}^*\mathbb{B}$, then $\mathbb{A}\times\mathbb{B}$ is a retract of its enlargement, as well.
\end{proposition}
\begin{proof}
Because ${}^*(\A\times\B) = {}^*\A\times{}^*\B$, a retract can be constructed entrywise.
\end{proof}

\begin{proposition}\label{}
If $\mathbb{A}\leq\mathbb{C}$ is a retract of ${}^*\mathbb{A}$ and $\mathbb{B}\leq\mathbb{C}$ is a retract of ${}^*\mathbb{B}$, then $\mathbb{A}\cap\mathbb{B}$ is a retract of its enlargement, as well.
\end{proposition}
\begin{proof}
Because ${}^*(\A\cap\B) = {}^*\A\cap{}^*\B$, a restriction of either retract can be applied to the intersection of the enlargements.
\end{proof}

\begin{proposition}\label{EnlargementCounterEx}
If $\mathbb{A}$ is a retract of ${}^*\mathbb{A}$ and $\mathbb{B}$ is a retract of ${}^*\mathbb{B}$, then $\mathbb{A}\oplus\mathbb{B}$ is not necessarily a retract of its enlargement.
\end{proposition}
\begin{proof}
Taking the direct sum of two finite groups in the variety of groups, which is infinite, shows us that the enlargement of the direct sum is not the direct sum of the enlargements. In the direct sum of two singleton semigroups, which is a subreduct of the direct sum of two groups, for $x \neq a,b, \eta \notin \N$, we see that $r(x^\eta)=r(x^{\eta-n})(x^n) \forall n\in \N$, producing a contradiction.
\end{proof}

\begin{proposition}\label{}
If $\mathbb{A}$ is a retract of $\widehat{\mathbb{A}}$ and $\mathbb{B}$ is a retract of $\widehat{\mathbb{B}}$, then $\mathbb{A}\times\mathbb{B}$ is a retract of its extension monad, as well.
\end{proposition}
\begin{proof}
Since the extension monad of a product is the product of the extension monads, we can simply use an entrywise retract.
\end{proof}

\begin{proposition}\label{}
If $\mathbb{A}\leq\mathbb{C}$ is a retract of $\widehat{\mathbb{A}}$ and $\mathbb{B}\leq\mathbb{C}$ is a retract of $\widehat{\mathbb{B}}$, then $\mathbb{A}\cap\mathbb{B}$ is a retract of its extension monad, as well.
\end{proposition}
\begin{proof}
The restriction of either retraction to the extension monad of the intersection is a retraction onto the intersection because it preserves operations and is the identity on the standard intersection.
\end{proof}

\begin{proposition}\label{}
If $\mathbb{A}$ is a retract of $\hat{\mathbb{A}}$ and $\mathbb{B}$ is a retract of $\hat{\mathbb{B}}$, then $\mathbb{A}\oplus\mathbb{B}$ is not necessarily a retract of its extension monad.
\end{proposition}
\begin{proof}
The example for the enlargements in \ref{EnlargementCounterEx} applies to extension monads here, as well, because the example summands were also finite in that example.
\end{proof}

Now we look into specific examples in which certain algebras are retracts of their enlargements and extension monads, particularly integers under different collections of operations.

\begin{example}\label{}
The extension monad of ring $\mathbb{Z}$, the integers, is a retract of ${}^*\mathbb{Z}$. Because the ring of integers is finitely generated, $\hat{\Z} = {}^*\Z$, and thus, we have a trivial retract.
\end{example}

\begin{example}\label{} 
The extension monad of the (additive, multiplicative) semigroup $\mathbb{Z}$ of integers is a retract of ${}^*\mathbb{Z}$. Due to finite generation of the additive multiplicative semigroup of integers, we get a trivial retract.
\end{example}

\begin{example}\label{} 
The extension monad of the lattice $\mathbb{N}$ of nonnegative integers (under the usual meet-and-join operations) is not a retract of ${}^*\mathbb{N}$. Because the lattice of nonnegative integers is locally finite, we have $\hat{\N} - \N$. Taking any infinite element $\psi \in {}^*\N$, we see that $\forall n \in \N$, we have $n \vee \psi = \psi$ and $n \wedge \psi = n$. However, to be consistent with the definition of a homomorphism, we need $f(\psi) = n_\psi \in \N$ to be such that $\forall n \in \N$ we have $n_\psi \vee n = n_\psi$ and $n_\psi \wedge n = n$, producing a contradiction.
\end{example}

\begin{example}\label{}
The extension monad of the lattice $\mathbb{N}$ of nonnegative integers (with meet and join based upon the divisibility ordering) is not a retract of ${}^*\mathbb{N}$. This structure is also locally finite, so $\hat{\N} = \N$. Similar to the previous example, taking some infinite $\psi \in {}^*\N$, we have that $f(2^\psi) > 2^n$ $\forall n \in \N$, giving us a contradiction.
\end{example}

\begin{example}\label{} 
The ring of integers is not a retract of its extension monad. Due to finite generation, we have that $\hat{\Z} = {}^*\Z$. Taking any infinite element $\eta$ gives us that  $\forall n\in \N$ we have $0 \neq r(\eta-n) = r(\eta) - n$, so $r(\eta) \neq n$, again producing a contradiction.
\end{example}

\begin{example}\label{}
The semigroup of integers is not a retract of its extension monad. Due to finite generation, we have that $\hat{\Z} = {}^*\Z$. At which point, this follows the same logic as the previous example.
\end{example}

\begin{example}\label{}
The lattice of nonnegative integers (with meet and join based on the usual ordering) is a retract of its extension monad. Due to local finiteness, we have a trivial retract.
\end{example}

\begin{example}\label{}
The lattice of nonnegative integers (with meet and join under the divisibility ordering) is a retract of its extension monad. Due to local finiteness, we have a trivial retract.
\end{example}
\newline

Given the change in the ordered set of operations in a superstructure enlargement, algebras are only directly comparable to subalgebras of their enlargement if they have a finite collection of finitary operations. Therefore, for any homomorphism to exist from the enlargement to the original algebra, the collection of algebras must be finite and consist only of finitary operations.

\begin{proposition}
    If $\A$ is a retract of $\hat{\A}$, then $\mathcal{O}_{\A}$ is a finite collection of finitary operations.
\end{proposition}
\begin{proof}
    The proof of this is trivial.
\end{proof}

The next several propositions assume that some retraction exists between an enlargement or extension monad and the original algebra, and then we examine the properties of different restrictions on those retractions.

\begin{proposition}\label{Part1}
Let $F$ be a finite subset of ${}^*A$ and assume that $\mathbb{A}$ is a retract of $\left<A\cup F\right>^{(int)}$ via retraction $r:\left<A\cup F\right>^{(int)}\twoheadrightarrow\mathbb{A}$. Then, the restriction of $r$ to $\left<F\right>^{(int)}$ must be a retraction onto its range.
\end{proposition}
\begin{proof}
Because the restriction of the retraction maps onto its range and the entire function is the identity with regards to the range, a restriction will have the same property.
\end{proof}

\begin{proposition}\label{}
Let $r:\widehat{\mathbb{A}}\twoheadrightarrow\mathbb{A}$ be a retraction and let $H\subseteq\widehat{A}$ be a hyperfinite set. The restriction of $r$ to $H$ does not need to be an internal mapping.
\end{proposition}
\begin{proof}
Let $\A$ be the natural numbers as a ring, $r:\widehat{\mathbb{A}}\twoheadrightarrow\mathbb{A}$ be a retraction, and $H\subseteq\widehat{A}$ be any hyperfinite set that contains $\A$. Since the ranges are not internal, $r$ and $r|_H$ are not internal functions.
\end{proof}

\begin{proposition}\label{}
Let $F$ be a hyperfinite subset of ${}^*A$ and assume that $\mathbb{A}$ is a retract of $\left<A\cup F\right>^{(int)}$ via retraction $r:\left<A\cup F\right>^{(int)}\twoheadrightarrow\mathbb{A}$. The restriction of $r$ to $\left<F\right>^{(int)}$ is a retraction onto its range.
\end{proposition}
\begin{proof}
This follows the same logic as \ref{Part1}.
\end{proof}

\begin{proposition}
If any of the following are true, then $\A$ is a retract of $\hat{\A}$:
\begin{itemize}
    \item $\A = \hat{\A}$, so $\A$ is locally finite
    \item There exists an internal subalgebra $\S \leq {}^*\A$ such that $\A$ is a retract of $\S$.
\end{itemize}
\end{proposition}
\begin{proof}
    The first statement is true due to a trivial retraction, while the second is true due to a restriction of the retract of the internal algebra.
\end{proof}

The following theorem is not only an interesting example of how hyperfinite iterations can break retractions, but it also gives us an easy example of algebras that are retracts of their enlargements but whose direct sums do not have the same property.

\begin{theorem}\label{NoFreeSemiGroup}
    No free semigroup is a retract of its enlargement.
\end{theorem}
\begin{proof}
    Given a free semigroup $\S$ with binary operation $*$ freely generated by a set of elements $\{g_i|i \in I\}$, then for some $\eta \in {}^*\N \setminus \N$ and fixed $j \in I$, we have $g_j^\eta \in {}^*\S$. We also know that, since $\S$ is free, there are no two elements $x,y \in S$ such that $x * y = g_i$ for any $i \in I$. Now, given retraction $r: {}^*\S \rightarrow \S$, we can see by associativity that there exists a finite sequence of generators, $a_1,a_2,...,a_n$, such that $r(g_j^\eta) = a_1 * a_2 * ... *a_n$. Furthermore, $\forall n\in \N$, we have $r(g_j^\eta) = r(g_j^n)*r(g_j^{\eta-n}) = g_j^n*r(g_j^{\eta-n})$. Therefore, because $g_j*r(g_j^{\eta-1}) = a_1 * a_2 * ... *a_n$ and because $g_j$ and $a_1$ are generators, we have $a_1 = g_j$. The same is true for each $a_k$ for $1 \leq k \leq n$. Hence, we can say that $r(g_j^\eta) = g_j^n$. However, we also know that $g_j^{n-1}*g_j*r(g_j^{\eta-n}) = g_j^{n-1} * g_j$, so by the cancellative property of free semigroups \cite[p.367, Chap IX, 5.6 ($\beta$)]{LyapinFreeCancellative}, we conclude that $g_j*r(g_j^{\eta-n}) = g_j$, giving us a contradiction.
\end{proof}

\begin{corollary}
    It is not necessarily the case that $\A\oplus\B$ is a retract of ${}^*(\A\oplus\B)$ or $\widehat{\A\oplus\B}$.
\end{corollary}
\begin{proof}
    The direct sum of two singleton semigroups, $\A$ and $\B$, is the free semigroup generated by two elements, which, given \ref{NoFreeSemiGroup}, shows that $\A\oplus\B$ is not a retract of ${}^*(\A\oplus\B) = \widehat{\A\oplus\B}$, which are equal due to finite generation.
\end{proof}

The following result serves to greatly restrict retractions between extension monads or enlargements and their standard algebras. In fact, the previous theorem is a direct result of the following.


\begin{theorem}\label{}
    Any algebra that is a retract of its extension monad or enlargement is finite.
\end{theorem}
\begin{proof}
The result comes almost immediately from a saturation argument.
Let $r: \hat{\A} \rightarrow \A$ be a retraction and take $\mathcal{F}$ to be the collection of restrictions of $r$ to finite subsets of $\hat{A}$. Now define relation $s$ such that for $r|_F, r|_G \in \mathcal{F}$, we have $(r|_F,r|_G) \in s$ if and only if $F \subseteq G$ and $r|_F,r|_G$ are internal functions that are partial homomorphisms. Each element of $\mathcal{F}$ is an internal function, because it has finite domain and range, and a partial homomorphism. We know there exists a finite set that contains any finite collection of finite sets. Therefore, by saturation, we can also find internal function $r|_I \in {}^*\mathcal{F}$ for internal set $I$ such that $\hat{A} \subseteq I$ and $r|_I$ is an internal partial homomorphism. Since $r|_I: I \twoheadrightarrow A$ is internal, we have that $A$ is internal and, thus, that $A = {}^*A = \hat{A}$, showing us that $A$ is finite. The same argument works if we assume instead that $A$ is a retract of ${}^*A$.
\end{proof}


\section{Future Work}
In conclusion, we include several problems and partial proofs related to the material presented in the body of this paper. These will be worked on in more detail in the first author's dissertation \cite{Danielle}.

\begin{problem}\label{}
    Is the following statement true?
    
    If $\A = (A; f, g)$, $\B = (A; f)$, and $\C = (A; g)$, and if $\hat{B} = \hat{C}$, then $\hat{A} = \hat{B} = \hat{C}$. Or more generally, if $\A = (A; \mathcal{O})$ and $\A_i = (A; \mathcal{O}_i)$ for $i \in I$ where $\displaystyle\bigcup_{i \in I} \mathcal{O}_i = \mathcal{O}$, and if $\hat{A_i} = \hat{A_j}$ for $i,j \in I$, then $\hat{A} = \hat{A_i}$ for all $i \in I$.
\end{problem}

\begin{problem}\label{}
    Is the following statement true?
    
    Assume that $\A$ is a retract of $\hat{\A}$ with $r: \hat{A} \rightarrow A$ satisfying the properties of retraction and assume that $\A \neq \hat{\A}$. Let $H$ be a finite subset of $A$. Then $\left < r^{-1}[H] \right >^{(int)} \geq \widehat{\left < H \right >}$.
\end{problem}
\newline

The behaviors of the extension monad of a ring intuitively seem as though they should follow logic similar to the case of the ring of integers, and the following conjectural property also seems as though it should extend to dibinary distributive algebras.

\begin{conjecture}\label{}
    The extension monad of a ring has the same underlying set as the extension monad of its additive reduct.
\end{conjecture}

\begin{conjecture}\label{}
    The previous property holds for all dibinary algebras\textbf{---}algebras with only two binary operations $\A = (A; \beta, \delta)$ with a distributive property, i.e. $\delta(x,\beta(y,z)) = \beta(\delta(x,y),\delta(x,z))$ and $\delta(\beta(x,y),z)=\beta(\delta(x,z),\delta(y,z))$\textbf{---}for the $\beta$ reduct.
\end{conjecture}
\newline

Note that the previous statement does not require both operations to distribute over each other; only one needs to distribute over the other.

\begin{problem}\label{}
    Is the following statement true?
    
    Let $\A, \B$ be algebras in variety $V$, then the $V$-direct sum satisfies:
    
    $\hat{\A} \oplus \hat{\B} \leq \widehat{\A \oplus \B} = \bigcap \{ \S_\A {}^*\oplus \S_\B | \S_\A, \S_\B \text{ internal, } \A \leq \S_\A \leq {}^*\A, \B \leq \S_\B \leq {}^*\B \}$.
    
    More generally, if $\{\A_i\}$ is an arbitrary collection of algebras in variety $V$, then the $V$-direct sum satisfies:
    
    $\displaystyle\bigoplus_{i \in I} \widehat{\A_i} \leq \widehat{\displaystyle\bigoplus_{i \in I} \A_i} = \bigcap \left \{ \displaystyle{}^*\bigoplus_{i \in I} \S_i | \S_i \text{ internal, } \A_i \leq \S_i \leq {}^*\A_i \right \}$.
\end{problem}

\begin{problem}\label{ProdMonad}
    Is the following statement true?
    
    The extension monad of a direct product equals the direct product of extension monads of the factors.
\end{problem}
\begin{proof}
    The following is a partial proof. Let $\{\A_i \}_{i \in I}$ be a collection of algebras with the same language. Then for any $i_* \in {}^*I \setminus I$, there exists an internal algebra $\B$ for which $\prod_{i \in I} \A_i$ is a subreduct of $\B$ and $\B$ does not have an entry correpsonding to $i_*$. Thus, $\widehat{\prod_{i \in I} \A_i}$ only has the standard entries corresponding to $I$. Each of these entries corresponds to some $\A_i$ for $i \in I$. If we let $\B_i \leq {}^*\A_i$ be a corresponding internal algebra such that $\A_i$ is a subreduct of $\B_i$, we see that ${}^*\A_1 \prod {}^*\A_2 \prod ... \prod \B_i \prod {}^*\A_{i + 1} \prod ...$ for $i \in {}^*I$ is an internal subalgebra for which $\prod_{i \in I} \A_i$ is a subreduct of this algebra. Therefore, we see that the algebra, which is simply ${}^*(\prod_{i \in I} \A_i)$ but $\hat{\A}_i$ at entry $i$, contains $\widehat{\prod_{i \in I} \A_i}$. Thus, $\widehat{\prod_{i \in I} \A_i} \subseteq \prod_{i \in I} \hat{\A}_i$.
\end{proof}



As can be seen by the number of unanswered problems, there is still significant work to be done on this topic. We plan to look deeper into the relations of direct sums and their extension monads, possibly touching on other research in related areas, as well. Furthermore, additional types of algebras that could be of interest include congruence commutative algebras and their associated varieties, periodic algebras, and semigroups of congruences. We would also like to consider the categorical structures and properties of extension monads, specifically if they are limits of their internal superalgebras in a categorical sense.

In future work, we also wish to broaden our scope to investigate reducts of retracts for the cases with infinitely many operations, as well as clarifying when an algebra is a retract of its enlargement.

\section{References}
\printbibliography

\end{document}